\documentclass[12pt]{article}
\usepackage{standalone}
\usepackage{amsmath, amsthm, amssymb, color}
\usepackage[colorlinks=true,linkcolor=blue,urlcolor=blue]{hyperref}
\usepackage{graphicx}
\usepackage{caption}
\usepackage{mathtools}
\usepackage{enumerate}
\usepackage{verbatim}
\usepackage{tikz,tikz-cd,tikz-3dplot}
\usepackage{amssymb}
\usetikzlibrary{matrix}
\usetikzlibrary{arrows}
\usepackage{algorithm}
\usepackage[noend]{algpseudocode}
\usepackage{caption}
\usepackage[normalem]{ulem}
\usepackage{subcaption}
\usepackage{multicol}
\oddsidemargin \evensidemargin
\usepackage{makecell}
\usepackage{array}
\usepackage{enumitem}
\usepackage{tcolorbox}
\usepackage{fancyvrb}
\usepackage{tikz}
\usetikzlibrary{positioning,fit,calc}

\newtheorem{theorem}{Theorem}[section]

\newtheorem{proposition}[theorem]{Proposition}
\newtheorem{lemma}[theorem]{Lemma}
\newtheorem{corollary}[theorem]{Corollary}

\theoremstyle{definition}

\newtheorem{remark}[theorem]{Remark}

\newtheorem{examplex}[theorem]{Example}
\newenvironment{example}
{\pushQED{\qed}\examplex}
{\popQED\endexamplex}

\newcommand{\bR}{\mathbb R}
\newcommand{\bC}{\mathbb C}
\newcommand{\bZ}{\mathbb Z}
\newcommand{\bN}{\mathbb N}

\newcommand{\bG}{\mathbb G}
\newcommand{\bP}{\mathbb P}

\newcommand{\cC}{\mathcal C}

\newcommand{\cH}{\mathcal H}

\renewcommand{\phi}{\varphi}

\newcommand{\codim}{{\rm codim}}

\usepackage{xcolor}

\title{\bf Multigraded Hurwitz forms}

\author{Elizabeth Pratt, Luca Sodomaco and Bernd Sturmfels}
\date{}

\begin{document}

\maketitle

\begin{abstract} 
\noindent
The Hurwitz form of a projective variety
characterizes linear spaces of complementary dimension
which meet the variety non-transversally.
We extend this  notion to varieties in a product of projective spaces.
This parallels the multigraded Chow forms due to Osserman and Trager.
 We study the degrees of multigraded Hurwitz forms.
 An explicit degree formula is given for complete intersections.
This offers a new tool for elimination theory that has
many applications, ranging from Nash equilibria to Feynman integrals.
\end{abstract}

\section{Introduction}\label{sec:Introduction}

Consider the following system of four equations in four unknowns $x,y,z,w$:
\begin{equation}
\label{eq:fourequations}
\begin{matrix}
 a_{11} x y+a_{12} z+a_{13} w+a_{14}  &=&   b_{11} zw+b_{12} x+b_{13} y+b_{14} & = &  0, \\
 a_{21} x y+a_{22} z+a_{23} w+a_{24}  &=&   b_{21} zw+b_{22} x+b_{23} y+b_{24} & = & 0.
 \end{matrix}
 \end{equation}
 If the $16$ coefficients $a_{ij},b_{ij}$ are generic, then (\ref{eq:fourequations}) has four distinct complex solutions.
 Two of these solutions come together when a discriminant
with  $10156 $ terms of degree $24$ vanishes.
Using the symbols $[ij] = a_{1i} a_{2j} - a_{1j}a_{2i}$ and $\langle ij \rangle = b_{1i} b_{2j} - b_{1j}b_{2i}$ for $2 \times 2$ minors, it equals
\[
\resizebox{\textwidth}{!}{$
\begin{matrix}
 16 \,[ 12 ]^3 [ 13 ]^3 \langle 14 \rangle^4 \langle 24 \rangle \langle 34 \rangle
-128\, [ 12 ]^2 [ 13 ]^3 [ 24 ] \langle 13 \rangle^2 \langle 14 \rangle^2 \langle 24 \rangle^2
+ 128 \,[ 12 ]^2 [ 13 ]^3 [ 24 ] \langle 13 \rangle \langle 14 \rangle^3 \langle 23 \rangle \langle 24 \rangle \\
+ 16 [ 12 ]^2 [ 13 ]^3 [ 24 ] \langle 14 \rangle^4 \langle 23 \rangle^2
-8 [ 12 ]^2 [ 13 ]^2 [ 14 ]^2 \langle 13 \rangle \langle 14 \rangle^4 \langle 24 \rangle
+ 128 [ 12 ] [ 13 ] [ 14 ]^2 [ 23 ]^2 \langle 12 \rangle \langle 13 \rangle^2 \langle 14 \rangle \langle 23 \rangle \langle 24 \rangle \\
+ 64 [ 12 ]^2 [ 13 ]^2 [ 14 ] [ 23 ] \langle 13 \rangle^2  \langle 14 \rangle^2 \langle 24 \rangle^2 
-64 [ 12 ]^2 [ 13 ]^2 [ 14 ] [ 23 ] \langle 13 \rangle \langle 14 \rangle^3 \langle 23 \rangle \langle 24 \rangle
-8 [ 12 ]^2 [ 13 ]^2 [ 14 ] [ 23 ] \langle 14 \rangle^4 \langle 23 \rangle^2 \\
+ 256 [ 12 ]^2 [ 13 ]^2 [ 23 ]^2 \langle 13 \rangle^3 \langle 24 \rangle^3
-384 [ 12 ]^2 [ 13 ]^2 [ 23 ]^2 \langle 13 \rangle^2 \langle 14 \rangle \langle 23 \rangle \langle 24 \rangle^2
+ 120 [ 12 ]^2 [ 13 ]^2 [ 23 ]^2 \langle 13 \rangle \langle 14 \rangle^2 \langle 23 \rangle^2 \langle 24 \rangle \\
+ 4 [ 12 ]^2 [ 13 ]^2 [ 23 ]^2 \langle 14 \rangle^3 \langle 23 \rangle^3
-128 [ 12 ] [ 13 ]^3 [ 24 ]^2 \langle 12 \rangle \langle 13 \rangle \langle 14 \rangle^2 \langle 23 \rangle^2
-256 [ 12 ] [ 13 ]^3 [ 24 ]^2 \langle 12 \rangle \langle 13 \rangle^2 \langle 14 \rangle \langle 23 \rangle \langle 24 \rangle \\
+ [ 12 ] [ 13 ] [ 14 ]^4 \langle 12 \rangle \langle 13 \rangle \langle 14 \rangle^4
+ 64 [ 12 ] [ 13 ]^2 [ 14 ]^2 [ 24 ] \langle 12 \rangle \langle 13 \rangle^2 \langle 14 \rangle^2 \langle 24 \rangle
-32 [ 12 ] [ 13 ]^2 [ 14 ]^2 [ 24 ] \langle 12 \rangle \langle 13 \rangle \langle 14 \rangle^3 \langle 23 \rangle \\
-144 [ 12 ] [ 13 ]^2 [ 23 ]^2 [ 24 ] \langle 12 \rangle \langle 13 \rangle \langle 14 \rangle \langle 23 \rangle^3 
+ 4 [ 12 ]^2 [ 13 ]^2 [ 14 ]^2 \langle 14 \rangle^5 \langle 23 \rangle
+ 288 [ 12 ] [ 13 ]^2 [ 23 ]^2 [ 24 ] \langle 12 \rangle \langle 13 \rangle^2 \langle 23 \rangle^2 \langle 24 \rangle  \\
+ 128 [ 12 ] [ 13 ]^2 [ 14 ] [ 23 ] [ 24 ] \langle 12 \rangle \langle 13 \rangle \langle 14 \rangle^2 \langle 23 \rangle^2
+ 256 [ 12 ] [ 13 ]^2 [ 14 ] [ 23 ] [ 24 ] \langle 12 \rangle \langle 13 \rangle^2 \langle 14 \rangle \langle 23 \rangle \langle 24 \rangle 
\\
+ 256 [ 12 ] [ 13 ]^3 [ 24 ]^2 \langle 12 \rangle \langle 13 \rangle^3 \langle 24 \rangle^2
\!-\!32 [ 12 ] [ 13 ] [ 14 ]^3 [ 23 ] \langle 12 \rangle \langle 13 \rangle^2 \langle 14 \rangle^2 \langle 24 \rangle
\!+\! 16 [ 12 ] [ 13 ] [ 14 ]^3 [ 23 ] \langle 12 \rangle \langle 13 \rangle \langle 14 \rangle^3 \langle 23 \rangle \\
-256 [ 12 ] [ 13 ]^2 [ 14 ] [ 23 ] [ 24 ] \langle 12 \rangle \langle 13 \rangle^3 \langle 24 \rangle^2 
-128 [ 12 ] [ 13 ] [ 14 ]^2 [ 23 ]^2 \langle 12 \rangle \langle 13 \rangle^3 \langle 24 \rangle^2
-27 [ 12 ] [ 13 ] [ 23 ]^4 \langle 12 \rangle \langle 13 \rangle \langle 23 \rangle^4 \\
-144 [ 12 ] [ 13 ] [ 14 ] [ 23 ]^3 \langle 12 \rangle \langle 13 \rangle^2 \langle 23 \rangle^2 \langle 24 \rangle
+ 256 [ 13 ]^3 [ 24 ]^3 \langle 12 \rangle^2 \langle 13 \rangle^2 \langle 23 \rangle^2
-62 [ 12 ] [ 13 ] [ 14 ]^2 [ 23 ]^2 \langle 12 \rangle \langle 13 \rangle \langle 14 \rangle^2 \langle 23 \rangle^2 \\
+ 72 [ 12 ] [ 13 ] [ 14 ] [ 23 ]^3 \langle 12 \rangle \langle 13 \rangle \langle 14 \rangle \langle 23 \rangle^3
-128 [ 13 ]^2 [ 14 ]^2 [ 24 ]^2 \langle 12 \rangle^2 \langle 13 \rangle^3 \langle 24 \rangle
+ 64 [ 13 ]^2 [ 14 ]^2 [ 24 ]^2 \langle 12 \rangle^2 \langle 13 \rangle^2 \langle 14 \rangle \langle 23 \rangle \\
-384 [ 13 ]^2 [ 14 ] [ 23 ] [ 24 ]^2 \langle 12 \rangle^2 \langle 13 \rangle^2 \langle 23 \rangle^2
-8 [ 13 ] [ 14 ]^4 [ 24 ] \langle 12 \rangle^2 \langle 13 \rangle^2 \langle 14 \rangle^2
+ 128 [ 13 ] [ 14 ]^3 [ 23 ] [ 24 ] \langle 12 \rangle^2 \langle 13 \rangle^3 \langle 24 \rangle \\
-64 [ 13 ] [ 14 ]^3 [ 23 ] [ 24 ] \langle 12 \rangle^2 \langle 13 \rangle^2 \langle 14 \rangle \langle 23 \rangle
+ 120 [ 13 ] [ 14 ]^2 [ 23 ]^2 [ 24 ] \langle 12 \rangle^2 \langle 13 \rangle^2 \langle 23 \rangle^2
-8 [ 14 ]^4 [ 23 ]^2 \langle 12 \rangle^2 \langle 13 \rangle^2 \langle 14 \rangle \langle 23 \rangle
\\
+ 4 [ 14 ]^5 [ 23 ] \langle 12 \rangle^2 \langle 13 \rangle^2 \langle 14 \rangle^2 
\!+\! 16 [ 14 ]^4 [ 23 ]^2 \langle 12 \rangle^2 \langle 13 \rangle^3 \langle 24 \rangle
\!+\! 16 [ 14 ]^4 [ 24 ] [ 34 ] \langle 12 \rangle^3 \langle 13 \rangle^3 
\!+\! 4 [ 14 ]^3 [ 23 ]^3 \langle 12 \rangle^2 \langle 13 \rangle^2 \langle 23 \rangle^2.
\end{matrix}
$}
\]
This expression has $45$ terms of bidegree $(6,6)$. It is a {\em multigraded Hurwitz form} for a 
 variety $X$ in $\bP^3 \times \bP^3$.
Namely, $X$ is the $4$-dimensional toric variety with monomial parametrization
\begin{equation}
\label{eq:toricvariety}
 \bC^4 \rightarrow \bP^3 \times \bP^3, \,\, (x,y,z,w) \mapsto \bigl(\,  [xy:z:w:1] \,,\,[zw:x:y:1]\, \bigr) . 
 \end{equation}
 The bihomogeneous prime ideal of $X$ is generated by the $2 \times 2$ minors of the $2 \times 3$ matrix
 \begin{equation}
 \label{eq:twobythree}
 \begin{bmatrix} \,x_0 y_3 & x_1 x_2 & x_3 \,\,\\
\,y_1 y_2 & x_3 y_0 & y_3 \,\,\end{bmatrix}.
\end{equation}
The system (\ref{eq:fourequations}) corresponds to
intersecting $X$ with a surface $ L_1 \times L_2 $.
Here $L_i$ is a line in $\bP^3$, given
by a point in the Grassmannian $\bG(1,3)$. Our polynomial defines the hypersurface
$$ \cH^{(1,1)}_X   = \,
\bigl\{ (L_1 , L_2) \, \in \, \bG(1,3) \times \bG(1,3) \,:\,
\hbox{the intersection} \,\,X \cap \,(L_1 \times L_2) \,\,\hbox{is not transverse} \bigl\}. $$
 We wrote the equation of $  \cH^{(1,1)}_X  $  in a notation that 
 hints at connections with physics \cite{HMPS2}.
 
 \smallskip
 
 The present article contributes to a long tradition in elimination theory.
Chow and van der Waerden defined the {\em Chow form} in
their 1937 article \cite{CW}. 
This extends work by Cayley from 1862.
Chow forms were picked up
by the computer algebra community in the late 20th century \cite{JKSS}.
A renewed focus on discriminants, through the book by
Gel'fand, Kapranov and Zelevinsky~\cite{GKZ}, led to the
introduction of the {\em Hurwitz form} in \cite{hurwitz}.
Osserman and Trager defined the {\em multigraded Chow form} in their
2019 article \cite{OT}.
We here complete the square:

\begin{figure}[ht]
\centering
\begin{tikzpicture}[
  node/.style={align=center, inner sep=2pt},
  arr/.style={->, line width=0.9pt, >=Stealth}
]

\node[node] (chow) {Chow form\\(1937)};
\node[node, below=2cm of chow] (mchow) {Multigraded\\Chow form\\(2019)};

\node[node, right=3.5cm of chow]  (hurwitz) {Hurwitz form\\(2017)};
\node[node, below=2cm of hurwitz] (mhurwitz) {Multigraded\\Hurwitz form\\(here)};

\draw[arr] (chow) -- (mchow);
\draw[arr] (hurwitz) -- (mhurwitz);

\draw[arr] (chow.east) -- (hurwitz.west);
\draw[arr] (mchow.east) -- (mhurwitz.west);

\node[fit=(chow)(mchow), inner xsep=20pt, inner ysep=12pt] (leftframe) {};
\draw[line width=0.9pt] (leftframe.north west) rectangle (leftframe.south east);

\node[fit=(hurwitz)(mhurwitz), inner xsep=20pt, inner ysep=12pt] (rightframe) {};
\draw[line width=0.9pt] (rightframe.north west) rectangle (rightframe.south east);

\end{tikzpicture}

\caption{The resultants (left) and discriminants (right) attached to a fixed variety $X$. \label{fig:22}}
\end{figure}
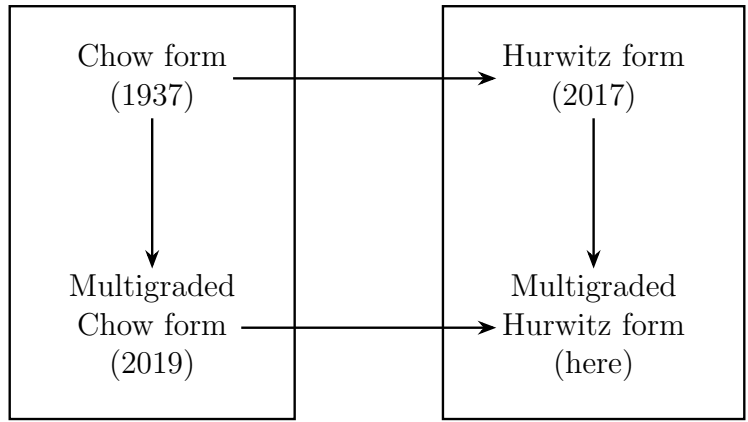

After the completion of this paper we learned that multigraded Hurwitz forms had
already been defined and explored by Dogan, Erg\"ur and Tsigaridas in \cite[Section 4.5]{DET}.

\smallskip

Our formula for  $\cH^{(1,1)}_X$ should be compared
with analogous data in the literature, for subvarieties in a single projective space.
For instance, for the Veronese surface in $\bP^5$, 
 the Hurwitz form is shown in \cite[Example 2.7]{hurwitz}, and the Chow form  is shown
in  \cite[Section 2.2]{sparse}. These two represent the classical tact invariant 
and the resultant of three ternary quadrics.
Multigraded Chow forms arising in computer vision were discussed in \cite[Example 5.14]{OT}.

An important feature of the two multigraded theories 
in Figure \ref{fig:22} is
that a fixed variety
has multiple Chow forms and multiple Hurwitz forms. They
are indexed by dimension~vectors. For instance, our toric $4$-fold $X \subset \bP^3 \times \bP^3$
has the following two multigraded Chow forms:
$$
\begin{matrix} &
\cC^{(0,1)}_X \,\, = \,\, \bigl\{
(p,L)  \in \bG(0,3) \times \bG(1,3) \,:\,
X \,\cap \, (p \times L) \,\not= \,\emptyset \bigr\}  \quad \hbox{has degree} \,\,(4,2)\,\, \, \smallskip \\
{\rm and} &
\cC^{(1,0)}_X \,\, = \,\, \bigl\{
(L ,p)  \in \bG(1,3) \times \bG(0,3) \,:\,
X \,\cap \, (L \times p) \,\not= \,\emptyset \bigr\}\quad \hbox{has degree} \,\,(2,4)  .
\end{matrix}
$$
The polynomial that defines the hypersurface $\cC^{(1,0)}_X$ is found to be
\[
\resizebox{\textwidth}{!}{$
\begin{matrix}
[24][34] \langle 123 \rangle^4 
+ [13][24] \langle 123\rangle^2 \langle 124 \rangle \langle 134  \rangle
+ [12] [34] \langle 123 \rangle^2 \langle 124 \rangle \langle 134  \rangle
+ [12] [13] \langle 124 \rangle^2 \langle 134\rangle^2 
- [23]^2 \langle 123 \rangle^3 \langle 234 \rangle.
\end{matrix}
$}
\]
Similarly, $X$ has two multigraded Hurwitz forms
$\cH^{(0,2)}_X$ and $\cH^{(2,0)}_X$. The latter equals
$$ \begin{small} \begin{matrix} [1]^2 \langle 124 \rangle^2 \langle 134\rangle ^2 
- 2 [1] [4] \langle 123 \rangle^2 \langle 124 \rangle \langle 134 \rangle
+ 4 [2] [3] \langle 123 \rangle^3 \langle 234  \rangle
+ [4]^2 \langle 123 \rangle^4. \end{matrix} \end{small} $$ This is the hypersurface
in $\bG(2,3) \times \bG(0,3)$
of  planes $P \times p$ that meet $X$ non-transversally.

\smallskip
We now discuss the organization and main contributions of this paper.
In Section \ref{sec2} we present formal definitions for our objects of study,
and we illustrate these with examples. Sparse systems like (\ref{eq:fourequations})
are featured in Examples \ref{ex:md} and \ref{ex:semimd}.
The ideal for the Hurwitz incidence is derived in
Proposition~\ref{prop:ideal of multigraded Hurwitz},
and from this one obtains the multigraded Hurwitz form ${\rm Hu}^\alpha_X$.
This is an irreducible polynomial in the
Pl\"ucker coordinates of several Grassmannians.

Section~\ref{sec3} starts out with two novel concepts, namely
multisectional genus and polynodal varieties.
Our main result (Theorem \ref{thm:degree})
 furnishes a formula for the degree of ${\rm Hu}^\alpha_X$.
This formula is exact for polynodal varieties, and it 
gives an upper bound 
in all other cases.

Generic complete intersections are treated in Section \ref{sec4}.
The degrees of their multigraded Hurwitz forms are expressed
in terms of the degrees of the defining equations (Theorem~\ref{thm:CIformula}).
We implemented our degree formulas in {\tt Macaulay2} \cite{M2}.
Our software {\tt MultiHurwitz.m2} is explained later in Section \ref{sec4}.
It is posted with detailed documentation 
on {\tt Zenodo} at \cite{zenodo}.
 
 The last two sections are devoted to varieties that are important in applications. In Section \ref{sec5}
  we examine toric varieties. Their Hurwitz forms
  are the mixed discriminants~\cite{CCDDS}. 
  Their degrees can be derived from volume polynomials \cite{Huh}
  and lattice polytopes   (Theorem~\ref{thm:khovanskii}).
  This is especially relevant for square
polynomial  systems with  only few distinct Newton polytopes.
  Such systems play a significant role in algebraic game theory~\cite{vectorBundleNash}.
We characterize games with degenerate Nash equilibria using 
multigraded Hurwitz forms (Proposition \ref{prop: multidegree games}).

  In Section \ref{sec6} we study varieties in a product of Grassmannians
  $\bG(1,3)$,   each sitting in its own Pl\"ucker embedding in $\bP^5$.
   Following \cite{HMPS1}, we consider
  incidence varieties of $\ell$ lines in $\bP^3$. These live
 in $\bG(1,3)^\ell$ and hence in $(\bP^5)^\ell$.
 Their multigraded Hurwitz forms are  the
LS discriminants in~\cite{HMPS2}.
Hence our theory is relevant for
  Landau analysis in particle physics. 

\section{Foundations and Examples}
\label{sec2}

Our ambient space is the product of projective spaces $\bP \coloneqq \bP^{n_1} \times \bP^{n_2} \times \cdots \times \bP^{n_\ell}$.
Note that $\bP$ has dimension $N:= n_1+n_2+\cdots+n_\ell$. 
The {\em Chow ring} of $\bP$ is the truncated polynomial~ring
\begin{equation}
\label{eq:chowring}
 A^*(\bP) \,\,=\,\, \bZ[T_1,T_2,\ldots,T_\ell] / \bigl\langle \,
T_1^{n_1+1},  \,T_2^{n_2+1},\,\ldots, \,T_\ell^{n_\ell+1} \bigr\rangle. 
\end{equation}
The coordinate
ring of $\bP$ is a polynomial ring $R$ in $N+\ell$ variables,  graded by the group $\bZ^\ell$.

We consider an irreducible variety $X$ of codimension $c$ in $\bP$.
Its defining prime ideal $I_X \subset R$ is $\bZ^\ell$-homogeneous.
The {\em multidegree} of $X$ is its class $[X]$ in the Chow ring $A^*(\bP)$.
This is a homogeneous polynomial of degree $c$ in the formal variables $T_1,T_2,\ldots,T_\ell$, written~as
\begin{equation}
\label{eq:multidegree}
\qquad    [X] \,\, = \,\, \sum_{|\alpha|=c} \delta_\alpha\,T^{\alpha} \qquad \in \quad A^*(\bP).
\end{equation}
In {\tt Macaulay2} \cite{M2},
one computes $[X]$ from generators of $I_X$  with the command
{\tt multidegree}.

Let $\bG(k,n)$ denote the Grassmannian of $(k+1)$-dimensional subspaces in $\bC^{n+1}$.
For us, points in $\bG(k,n)$ are linear subspaces of dimension $k$ in $\bP^n$.
In the multigraded setting, where the ambient space is $\bP$, the role of linear subspaces is
played by products $L_1 \times \cdots \times L_\ell$, where
$L_i \in \bG(\alpha_i,n_i)$ for $i$ in $ [\ell] := \{1,\ldots,\ell\}$.
This is a subvariety in $\bP$ of dimension $|\alpha| := \sum_{i=1}^\ell \alpha_i$.

\begin{example}[Toric $4$-fold] \label{ex:toric4fold}
We have $\ell=2,n_1=n_2=3$ for the variety $X$ in (\ref{eq:toricvariety}). From the
determinantal ideal $I_X$ in (\ref{eq:twobythree}), we compute the multidegree
$[X] =  2 T_1^2 + 4 T_1 T_2 + 2 T_2^2$. The middle coefficient $4$
says that the system (\ref{eq:fourequations}) has four solutions.
We shall see the effect of the coefficients $2,4,2$ on 
the bidegrees of all Chow and Hurwitz forms in the Introduction.
\end{example}

Osserman and Trager \cite{OT} defined the {\em multigraded Chow forms} 
 ${\rm Ch}^\alpha_X$ of the variety $X \subset \bP$ as follows.
 Fix a vector $\alpha = (\alpha_1, \, \ldots, \, \alpha_\ell) \in \bN^\ell$ such that
 $|\alpha| = c-1$ and $\alpha_i \leq n_i$ for $i \in [\ell]$.
Then ${\rm Ch}^\alpha_X$ is a polynomial in the coordinate ring $\bC[\bG_\alpha]$ of
the product of Grassmannians
\begin{equation}
\label{eq:Gdef}
 \bG_\alpha \,\,:= \,\,\bG(\alpha_1, n_1) \times \cdots \times \bG(\alpha_\ell,n_\ell).
 \end{equation}
 Elements $ (L_1,\ldots,L_\ell) \in \bG_\alpha$ correspond to subvarieties  
 $L = L_1 \times \cdots \times L_\ell$ of dimension $c-1\,$ in $\,\bP = \bP^{n_1}  \times \cdots \times \bP^{n_\ell}$.
 We are interested in the {\em multigraded Chow locus}
\begin{equation}
\label{eq:chowlocus} \mathcal{C}^\alpha_X \,\,\, := \,\,\,
  \bigl\{ (L_1,\ldots,L_\ell) \in \bG_\alpha\,:\, L \,\cap \, X \not= \emptyset \bigr\}.
  \end{equation}
 A combinatorial criterion for when (\ref{eq:chowlocus}) has codimension $1$ is given in
 \cite[Theorem 1.2]{OT}.  If this holds then the Chow form
     ${\rm Ch}^\alpha_X$ is the defining polynomial of (\ref{eq:chowlocus}).
 This is irreducible, and its degree in the
 Pl\"ucker coordinates of the $i$-th factor $\bG(\alpha_i, n_i)$ is
 the coefficient $\delta_{\alpha+e_i}$ of the multidegree $[X]$ in (\ref{eq:multidegree}).
 If $ \mathcal{C}^\alpha_X$ has codimension $\geq 2$  in $\bG_\alpha$, then 
 one sets ${\rm Ch}^\alpha_X := 1$.

 \smallskip
 
We now define the multigraded Hurwitz forms of  $X$. 
Each of them arises from projecting
an incidence correspondence. Fix $\alpha \in \bN^\ell$ such that $|\alpha| = c$ and $\alpha_i \leq n_i$. 
We define $L$ as before and $p=(p_1, \ldots, p_\ell)$ for $p_i \in \bP^{n_i}$.
For every subset $S\subset  [\ell]$, let $\pi_S$ 
denote the projection from $\bP$ to $\prod_{i\in S}\bP^{n_i}$.
The \emph{multigraded Hurwitz incidence} is the following Zariski closure:
\begin{equation}\label{eq:hurwitzincidence} \!\!
    \Phi^\alpha_X \,:=\,  \overline{\left\{(p,L)\ \bigg|\ 
    \begin{aligned}
        &\text{$L\cap X_{\rm reg}$ is not transverse at $p$, but}\\
        &\text{$\pi_S(L)\cap\pi_S(X)_{\rm reg}$ is transverse at $\pi_S(p)$ for $S\subsetneq [\ell]$}
    \end{aligned}\right\}} \,\subset \,X \times \bG_\alpha\,.
\end{equation}
Here $X_{\rm reg}$ denotes the set of regular points in the variety $X$.
We define the \emph{multigraded Hurwitz locus} of $X$ with respect to $\alpha$ to be the projection of $\Phi^\alpha_X$ to the second factor. We denote it by $\cH_X^\alpha$.  When this is a hypersurface, we call its defining polynomial the \emph{multigraded Hurwitz form} with respect to $\alpha$.
It is denoted by $\mathrm{Hu}_X^\alpha$. Otherwise, we set $\mathrm{Hu}_X^\alpha:= 1.$ 

The last condition in \eqref{eq:hurwitzincidence} is imposed so that the Hurwitz locus $\cH_X^\alpha$ is irreducible under suitable assumptions; see Theorem \ref{thm:degree}.
 In this case, $\mathrm{Hu}_X^\alpha$ is an irreducible polynomial in the Pl\"ucker coordinates on $\bG(\alpha_i,n_i)$,
and $\mathrm{Hu}_X^\alpha$ is homogenous in these unknowns for each $i \in [\ell]$.
Thus the degree of $\mathrm{Hu}_X^\alpha$ is a vector in $\bN^\ell$.
We shall study this degree in Section~\ref{sec3}.

\begin{example}[Plane curves] \label{ex:planecurves}
Fix a bivariate polynomial $f(x,y)$ which has degree $\geq 2$ in each variable $x$ and $y$.
Let $X$ denote the curve it defines in $\bP^1 \times \bP^1$. Here,
$\ell =2, n_1=n_2=1$ and $c=1$. The vector $\alpha$ is $(1,0)$ or $(0,1)$.
The two Hurwitz forms are the discriminants
\begin{equation}
\label{eq:twodiscriminants}
  \mathrm{Hu}_X^{(1,0)} \,= \, {\rm discr}_x \bigl(f(x,y) \bigr) 
 \qquad {\rm and} \qquad
  \mathrm{Hu}_X^{(0,1)} \,= \, {\rm discr}_y \bigl(f(x,y) \bigr) .
\end{equation} 
Indeed, in  $\Phi^\alpha_X$ we record tangency of the vertical lines
$L = \{x\} \times \bP^1$ and the horizontal lines 
$L = \bP^1 \times \{y\}$. The transversality
requirement is satisfied because $\pi_i(X) = \bP^1$ for $i=1,2$.
Note that $\bG_{(1,0)} = \bG_{(0,1)} = \bP^1$, so the discriminants
in (\ref{eq:twodiscriminants}) are binary forms. The Chow locus
is the curve itself, inside $\bG_{(0,0)} = \bP^1 \times \bP^1$, so
${\rm Ch}_X^{(0,0)}$ is the bihomogenization of $f(x,y)$.
\end{example}

\begin{remark} \label{rmk:reduciblehurwitz}
Example \ref{ex:planecurves} shows that
the Hurwitz form ${\rm Hu}^\alpha_X$ can be reducible.
This happens because  $n_i = 1$ for some $i  \in [\ell]$.
Indeed, the branch locus in $\bP^1$ consists of several points.
\end{remark}

\begin{example}[$\ell=1$]
Let $X$ be a variety of codimension $c$ and degree $d \geq 2$
in a single projective space $\bP^n$.
By \cite{hurwitz}, the Hurwitz form $\mathrm{Hu}^c_X$ is a polynomial
of degree $2d+2g-2$ in the Pl\"ucker coordinates of $\bG(c,n)$.
Here $g$ is the {\em sectional genus} of $X$, i.e.~the
genus of the curve
$X \cap L$ where $L \in \bG(c+1,n)$ is generic.
In Section \ref{sec3} we  shall define
a notion of sectional genus for our
multigraded setting.
This is needed for the degree of $\mathrm{Hu}^\alpha_X$ when $\ell \geq 2$.
The Chow form ${\rm Ch}^{c-1}_X$ is the classical Chow form \cite{CW, JKSS} of the
projective variety $X$.
\end{example}

\begin{example}[Mixed discriminants] \label{ex:md}
We revisit the definition in \cite{CCDDS}
which rests on \cite{GKZ}.
Here $A_1,\ldots,A_\ell$ are finite subsets of $\bZ^\ell$,
defining a square system of sparse Laurent polynomials.
The mixed discriminant for that system is a multigraded
Hurwitz form $\mathrm{Hu}_X^\alpha$.
Namely, $X$ is the  toric variety in $\bP$ given 
by $A_1,\ldots,A_\ell$. We have $n_i = |A_i| -1$ for $i \in [\ell]$,
and we set $\alpha = (n_1 - 1,\ldots,n_\ell - 1)$.
The Hurwitz locus $\cH_X^\alpha$ represents 
 polynomial systems \cite[eqn (1)]{CCDDS} that have
a non-degenerate multiple root. The notion of
non-degeneracy in \cite{CCDDS} matches the second point
in our definition of $\Phi^\alpha_X$. It is designed to  ensure that the mixed discriminant is irreducible.
A related concept of importance in computer algebra is the {\em sparse resultant}.
This is the multigraded Chow form \cite{OT} we obtain when taking
$\ell+1$ subsets $A_0,A_1,\ldots, A_\ell$ in $\bZ^\ell$. The irreducibility of 
 discriminants and resultants is a key point in GKZ theory \cite{GKZ}.
 \end{example}

The following generalization of Example \ref{ex:md} will become important for us in Section \ref{sec5}.

\begin{example}[Semi-mixed discriminants] \label{ex:semimd}
Suppose that $A_i = A_j$ for some $i \not= j$.
We aggregate all equations  with the same support. One 
refers to \cite[eqn (1)]{CCDDS} as a semi-mixed system of equations.
If the support $A_i$ occurs $k_i$ times, then $\mathcal{H}^\alpha_X$ lies in
a product of Grassmannians $\bG(n_i-k_i,n_i)$, and the Hurwitz 
form $\mathrm{Hu}^\alpha_X$ is expressed in their Pl\"ucker coordinates.
We saw this in the Introduction. The semi-mixed system
(\ref{eq:fourequations}) has $\ell=2$ and $X \subset (\bP^3)^2$.
This allowed us to write $\mathrm{Hu}^\alpha_X$ with only $45$ terms.
In the fully mixed setting of Example \ref{ex:md}, we would
work on a variety $X$ in $(\bP^3)^4$, so the mixed discriminant has 
$10156$ terms. If all supports $A_i$ are equal, then the
system is unmixed, and we are using a toric variety $X$
in a single $\bP^n$. Here $\mathrm{Hu}^c_X$ is the
Hurwitz form of a toric variety, as  in 
\cite[Examples 2.3 and 2.7]{hurwitz}.
\end{example}

We now return to the general case. Our  aim is to
describe the  ideal of the multigraded Hurwitz incidence $\Phi^\alpha_X$.
Let $I_X$ be the prime ideal of the irreducible variety $X \subset \bP$.
The generators of $I_X$ are $\bZ^\ell$-homogeneous polynomials in the ring $R\coloneqq\bigotimes_{i=1}^\ell \bC[x^{(i)}]$, where $\bC[x^{(i)}]$ denotes $\bC[x_{i,0},\ldots,x_{i,n_i}]$.  In what follows,
 we work in the coordinate ring $R[\bG_\alpha]$ of the product (\ref{eq:Gdef}) over $R$.
Let $J_i$ be the  bilinear ideal that encodes the condition $p_i\in L_i$. The~ideal
\[
    K\,\,\coloneqq \,\,I_X + J_1 + \cdots + J_\ell \quad \subset \,\, R[\bG_\alpha]
\]
defines the incidence variety $W\subset\bP \times  \bG_\alpha$ 
whose points are the pairs $(p,L)$ with $p\in L\cap X$. 

Let $I_{{\rm sing}(X)}$ be the ideal of the singular locus of $X$, and let $A$ be the Jacobian matrix of partial derivatives of the generators of $K$ with respect to all variables $x_{i,j}$ in $R$.
  The ideal
\[
    J \,\,\coloneqq \,\,\bigl(\,K + \,\langle\text{$\codim(W)$-minors of $A$}\rangle \,\bigr)\,\colon I_{{\rm sing}(X)}^\infty
\]
encodes the condition ``$L\cap X_{\rm reg}$ is not transverse at $p$'' in the definition of the variety $\Phi^\alpha_X$.

It remains to encode the requirement  ``$\pi_S(L)\cap\pi_S(X)_{\rm reg}$ is transverse at $p$''.
For $S\subsetneq [\ell]$, the projection $\pi_S(X)$ of our variety $X$ to $\prod_{i\in S}\bP^{n_i}$ is defined by the elimination ideal
\[ \mathfrak{I}_S \,\,\coloneqq \,\,I_X \,\cap\,\bigotimes_{i\in S}\bC[x^{(i)}]\,. \]
Let $W_S$ denote the incidence variety in $\prod_{i\in S}\bP^{n_i} \times \prod_{i\in S}\bG(\alpha_i, n_i)$
that is defined by the ideal
\[
    K_S\,\,\coloneqq \,\,\mathfrak{I}_S \,+\, \sum_{i\in S} J_i\,.
\]
Let $\mathfrak{I}_{S,\,{\rm sing}}$ be the ideal of the singular locus of $\pi_S(X)$,
and let $A_S$ be the Jacobian matrix of partial derivatives of the generators of $K_S$ with respect to 
the variables $x_{i,j}$ with $i\in S$.

The locus we wish to remove in the definition of $\Phi^\alpha_X$ is defined by the saturation ideal
\[
\qquad
    T_S\,\, \coloneqq \,\, \bigl(\,K_S \,+ \langle\text{$\codim(W_S)$-minors of $A_S$}\rangle\bigr)\,
    \colon \mathfrak{I}_{S,\,{\rm sing}}^\infty
\qquad {\rm for} \,\,\, S \subsetneq [\ell]\,.
\]
Indeed, the points in $T_S$ describe scenarios where
$\pi_S(L)$ intersects $\pi_S(X)_{\rm reg}$ non-transversally. 

\begin{proposition}\label{prop:ideal of multigraded Hurwitz}
The ideal that defines the multigraded Hurwitz incidence $\Phi^\alpha_X$ equals
\begin{equation}\label{eq:idealPhi}
    J\,\colon\,\biggl(\,\prod_{S\subsetneq [\ell]}T_S \biggr)^{\! \infty}.
\end{equation}
The ideal of the Hurwitz locus  $\cH_X^\alpha$ 
is obtained from (\ref{eq:idealPhi})
by eliminating all the variables $x_{i,j}$ in $R$.
If this elimination ideal is principal then its generator is the Hurwitz form
$\mathrm{Hu}^\alpha_X$.
\end{proposition}

\begin{proof}
The ideals $J,\mathfrak{I}_S,K_S,T_S$  in $R[\bG_\alpha]$
represent a translation into algebra of our geometric definition of
$\Phi^\alpha_X$ in (\ref{eq:hurwitzincidence}). The projection from $\Phi^\alpha_X$ onto $\cH_X^\alpha$ 
corresponds to the final elimination step, where we intersect (\ref{eq:idealPhi})
with $\bC[ \bG_\alpha]$. This ring is where the Hurwitz form lives.
\end{proof}


 Proposition \ref{prop:ideal of multigraded Hurwitz}
provides an algorithm for computing multigraded Hurwitz forms.
When implementing this, it is important to be
clever about all saturation steps, especially that
in (\ref{eq:idealPhi}). In practice, it is advantageous to
perform elimination prior to saturating, and to 
extract $\mathrm{Hu}^\alpha_X$ as the main factor from the output.
Here is an illustration for a small example.
 
\begin{example} \label{ex:hyperdeterminant}
 Let $X$ be the surface in $\bP^3\times\bP^3$ cut out by the nine quadratic binomials
\begin{align}\label{eq:9quadrics}
\begin{split}
    & y_{1}y_{2}-y_{0}y_{3},\quad x_{2}y_{2}+x_{3}y_{3},\quad x_{2}y_{0}+x_{3}y_{1},\\
    & x_{1}y_{1}+x_{3}y_{3},\quad x_{1}y_{0}+x_{3}y_{2},\quad x_{0}y_{2}+x_{1}y_{3},\\
    & x_{0}y_{1}+x_{2}y_{3},\quad x_{0}y_{0}-x_{3}y_{3},\quad x_{1}x_{2}-x_{0}x_{3}\,.
\end{split}
\end{align}
This is the conormal variety of the quadratic surface 
$Q =\{ x_1x_2 - x_0x_3=0\}$ in $\bP^3$.  Hence $\pi_1(X)=Q$.
The dual variety 
 $\pi_2(X)=Q^\vee$ in the second $\bP^3$ equals $\{y_{1}y_{2}-y_{0}y_{3} = 0\}$.
 Geometrically, the surface $Q$ is self-dual.
  The multidegree is $[X] \,= \,2\,T_1^3T_2 + 2\,T_1^2T_2^2 + 2\,T_1T_2^3$. 

We compute the locus $\cH_X^\alpha$ for $\alpha=(2,2)$. 
The Hurwitz incidence $\Phi^\alpha_X$ lives in $(\bP^3)^4$, with coordinates $[x_0:x_1:x_2:x_3]$, $[y_0:y_1:y_2:y_3]$, $[a_0:a_1:a_2:a_3]$ and $[b_0:b_1:b_2:b_3]$.
The last two points represent planes in $\bP^3$. The ideal $I_X$ is generated by \eqref{eq:9quadrics}. The condition $(p_1,p_2)\in L \cap X$ is encoded by the ideal $K = I _X+ J_1 + J_2 = I_X + \langle a_0x_0+\cdots+a_3x_3,b_0y_0+\cdots+b_3y_3\rangle$.
Once checks that $K$ is a prime ideal. The variety it defines has codimension $6$ in $(\bP^3)^4$. 

The Jacobian matrix of the $11$ generators of $K$ with respect to $x_0,\ldots,x_3,y_0,\ldots,y_3$ equals
\[
    A \,\,\,=\,\, \begin{footnotesize}
    \begin{bmatrix}
      0&0&0&0&-y_{3}&y_{2}&y_{1}&-y_{0}\\
      0&0&y_{2}&y_{3}&0&0&x_{2}&x_{3}\\
      0&0&y_{0}&y_{1}&x_{2}&x_{3}&0&0\\
      0&y_{1}&0&y_{3}&0&x_{1}&0&x_{3}\\
      0&y_{0}&0&y_{2}&x_{1}&0&x_{3}&0\\
      y_{2}&y_{3}&0&0&0&0&x_{0}&x_{1}\\
      y_{1}&0&y_{3}&0&0&x_{0}&0&x_{2}\\
      y_{0}&0&0&-y_{3}&x_{0}&0&0&-x_{3}\\
      -x_{3}&x_{2}&x_{1}&-x_{0}&0&0&0&0\\
      a_{0}&a_{1}&a_{2}&a_{3}&0&0&0&0\\
      0&0&0&0&b_{0}&b_{1}&b_{2}&b_{3}
    \end{bmatrix}. \end{footnotesize}
\]
After computing the ideal $J=K+\langle\text{$6\times 6$ minors of $A$}\rangle$, we need to avoid  that either $Q\cap L_1$ or $Q\cap L_2$ is a non-transverse intersection. These two conditions are encoded by the ideals
$$
\begin{matrix}
    T_1 &=& \langle x_{1}x_{2}-x_{0}x_{3},a_0x_0+\cdots+a_3x_3\rangle &+& \langle\text{$2\times 2$ minors of $A_1$}\rangle,\\
    T_2 &=& \langle \,y_{1}y_{2}-y_{0}y_{3}\,,b_0y_0+\cdots+b_3y_3\,\rangle &+& \langle\text{$2\times 2$ minors of $A_2$}\rangle,
\end{matrix}
$$
where
$
A_1 =
\begin{footnotesize}
\begin{bmatrix}
    -y_3 & y_2 & y_1 & \!\!-y_0\\
    \,b_0 & b_1 & b_2 & b_3
\end{bmatrix}\end{footnotesize} $ and $
A_2 = \begin{footnotesize}
\begin{bmatrix}
    -x_3 & x_2 & x_1 & \!\!-x_0\\
    \,a_0 & a_1 & a_2 & a_3
\end{bmatrix} \end{footnotesize}$.
The ideal of $\Phi^\alpha_X$ is $J\colon(T_1T_2)^\infty$. 
Eliminating $x_0,\ldots,x_3$ and $y_0,\ldots,y_3$ from the previous ideal yields the Hurwitz form
$$
\begin{matrix} {\rm Hu}^{(2,2)}_X \, &= &
    a_{0}^{2}b_{0}^{2}+a_{1}^{2}b_{1}^{2}+a_{2}^{2}b_{2}^{2}+a_{3}^{2}b_{3}^{2}
     +4\,a_{1}a_{2}b_{0}b_{3}
   +4\,a_{0}a_{3}b_{1}b_{2} 
     +2\,a_{0}a_{1}b_{0}b_{1} \\ & & 
     +\,2\,a_{0}a_{2}b_{0}b_{2}
      -2\,a_{0}a_{3}b_{0}b_{3}
    -2\,a_{1}a_{2}b_{1}b_{2}
      +2\,a_{1}a_{3}b_{1}b_{3}
   +2\,a_{2}a_{3}b_{2}b_{3}.
\end{matrix}
$$
This is the {\em hyperdeterminant} of the $2 \times 2 \times 2$ tensor
formed by the eight coordinates of $\bG_{(2,2)}$.
Note that the nine quadrics in (\ref{eq:9quadrics}) also define the
Segre embedding of $\bP^1 \times \bP^1 \times \bP^1$ in $\bP^7$.
\end{example}

Conormal varieties provide an interesting family to test
 our theory. Here is the general definition.
Given any irreducible variety $Z\subset\bP^n$, its \emph{conormal variety}  is $N_Z=\overline{N_Z^{\circ}}$, where
\[
    N_Z^{\circ}\,\coloneqq\,\bigl\{\,(p,H)\in\bP^n\times(\bP^n)^*\mid\text{
    $H$ is tangent to $Z$ at 
    $p\in Z_{\rm reg}$}\bigr\}\,.
\]
It is known that $N_Z$ is irreducible of codimension $n+1$ in $\bP^n\times(\bP^n)^*$. Its multidegree is
\[
    [N_Z] \,\,=\,\, \delta_0\,T_1^nT_2\,+\,\delta_1\,T_1^{n-1}T_2^2\,+\,\cdots\,+\,\delta_{n-1}\,T_1T_2^n\,,
\]
where the coefficients $\delta_i$ are the \emph{polar degrees} of $Z$. Reading from left to right, the first nonzero coefficient $\delta_i$ is the degree of the dual variety $Z^\vee$ which has codimension $i-1$ in $(\bP^n)^*$.
In Example~\ref{ex:hyperdeterminant}
we studied $N_Z$ for the surface $Z$ defined by the
$2 \times 2$ determinant.

\section{Degree Formula}
\label{sec3} 

In this section we present a formula for the degree of the multigraded Hurwitz form.
To state our result, we need some definitions. As before, we fix an irreducible
variety $X$ of codimension $c\,$ in $\,\bP = \bP^{n_1}  \times \cdots \times \bP^{n_\ell}$.
The \emph{multisectional genus} of $X$ is a polynomial
\begin{equation} \label{eq:msg}
 g(X)\,\,\,\, \coloneqq \,\,\sum_{|\beta| = c+1} \!
 \! g_{\beta}\,T^\beta. 
 \end{equation}
The coefficient $g_\beta$ is the geometric genus of the curve $X\cap (L_1\times\cdots\times L_\ell)$,
where $L_1 \in \bG(\beta_1,n_1),\, \ldots, \, L_\ell \in \bG(\beta_\ell,n_\ell)$ are generic.
If that intersection is not a curve, then  $g_\beta\coloneqq 0$.

\begin{example}
We consider three subvarieties in $\bP = \bP^3 \times \bP^3$.
First, let $X_2$ be the conormal surface of the Fermat cubic $C =\{x_0^3 + x_1^3 + x_2^3 + x_3^3 = 0\}$.
This has multidegree $12 T_1^3 T_2 + 6 T_1^2 T_2^2 + 3 T_1 T_2^3$, and its ideal has $22$ minimal
generators. The multisectional genus 
 records the sectional genera  of the dual surface $C^\vee$ of degree $12$ and
  of the given cubic surface $C$:
$$\, g(X_2) \,\,= \,\,4 T_1^3 T_2^2 \,+\, T_1^2 T_2^3.$$

Let $X_3$ be the intersection of
three generic hypersurfaces of degrees $(1,2),(1,2),(2,1)$.
The multisectional genus of the threefold $X_3$ can be computed using the adjunction formula:
$$ \, g(X_3) \,\,=\,\, 4 T_1^3 T_2 \,+\,16 T_1^2 T_2^2 \,+\, 11 T_1 T_2^3. $$
The middle coefficient $16$ is the genus of $X_3 \cap (P_1\! \times \! P_2)$ where 
$P_1,P_2$ are generic planes in~$\bP^3$.

Finally, let $X_4$ be the toric $4$-fold in Example \ref{ex:toric4fold}.
Its multisectional genus equals
\begin{equation}
\label{eq:genusoneone}
g(X_4) \,\, = \,\,  T_1^2 T_2 \,+\, T_1 T_2^2. 
\end{equation}
This means that the curve in $\bP$ defined by three of the
four equations in (\ref{eq:fourequations}) has genus one. 
\end{example}

For arbitrary subvarieties $X$ of $\bP$, we compute
the multisectional genus $g(X)$ as follows. 
From the ideal of the curve $C = X \cap (L_1\times \cdots \times L_\ell)$, we compute the multigraded Hilbert~series 
$$ \frac{K(z_1,z_2,\ldots,z_\ell) }{(1-z_1)^{n_1+1} (1-z_2)^{n_2+1} \,\cdots\, (1-z_\ell)^{n_\ell+1}}
\quad = \quad \sum_{u \in \bN^\ell} c_{u_1,u_2,\cdots,u_\ell} z_1^{u_1} z_2^{u_2} \cdots z_\ell^{u_\ell}. $$
Next, we examine the sequence of diagonal coefficients
$c_{m,m,\ldots,m}$ for $m =1,2,3,\ldots$.
For $m \gg 0$, this sequence is a linear function of $m$, namely it is the
Hilbert polynomial of the Segre embedding 
of $C \subset \bP$ into a projective space of
dimension $\prod_{i=1}^\ell (n_i+1) - 1$. Therefore,
\begin{equation}
\label{eq:getgenus} \quad
 c_{m,m,\ldots,m} \,\, = \,\, {\rm degree}(C) \cdot m \,+\, 
(1-{\rm genus}(C)) \qquad {\rm for} \,\,\,\, m \gg 0.
\end{equation}
If $C$ is smooth then this gives us the correct geometric genus for
the coefficient of $T^\beta$ in $g(X)$.
If $C$ is nodal then we subtract the number of nodes.
Otherwise, we just get a bound.

We now come to the second ingredient for our degree formula
in Theorem \ref{thm:degree}.
The following definition
 allows us to weaken the hypothesis, made in \cite[Theorem 1.1]{hurwitz},
that the singular locus of $X$ has codimension at least~$2$.
A curve $C \subset \bP$ is said to be \emph{polynodal} if, for any subset $S\subseteq [\ell]$, the projection of $C$ to $\prod_{i\in S}\bP^{n_i}$ is a curve with at most nodal singularities. Our variety $X \subset \bP$  is called \emph{polynodal} if, for any generic linear variety $L_1\times\cdots\times L_\ell$
of dimension $c+1$ in $\bP$, the curve $X\cap (L_1\times\cdots\times L_\ell)$ is polynodal,
whenever this is a curve.

Note that to check whether a curve is polynodal, we do not have to check the projection to every subproduct of projective spaces.
It suffices to check that the projections to every $\bP^{n_i}$ with $n_i \geq 2$ and every $\bP^{n_i} \times \bP^{n_j}$ with $n_i = n_j = 1$ have at most nodal singularities.

\begin{remark}
Consider the case $\ell=1$ when $X$
is a variety of degree $d$ in $\bP^n$. Then $X$ being polynodal means that 
the generic curve section of $X$ has at most nodal singularities.
This ensures that the Hurwitz form
${\rm Hu}^c_X$ has the expected degree $2d+2g-2$, where $g$ is the
geometric genus of that curve. In \cite[Theorem 1.1]{hurwitz},
the stronger assumption was made that $X$ is regular in codimension one, in
which case geometric genus equals  arithmetic genus.
\end{remark}

\begin{example} Conormal varieties are generally not polynodal.
For instance, the conormal curve $C \subset \bP^2 \times \bP^2$
of any smooth curve of degree $\geq 3$ in $\bP^2$ is not polynodal
because the inflection points of  the given curve correspond to cusps in its dual curve. The conormal surface $X \subset \bP^3 \times \bP^3$ in Example
\ref{ex:hyperdeterminant} is polynodal because it is derived
from a quadric.
\end{example}

We now state our main result. It generalizes \cite[Theorem 1.1]{hurwitz} to the multigraded case.
For the irreducibility result, the hypothesis on the dimensions $n_i$ is needed by
Remark \ref{rmk:reduciblehurwitz}.

\begin{theorem} \label{thm:degree}
	Let $X$ be an irreducible variety of codimension $c$ in $\bP$. 
	Fix $\alpha  \in \bN^\ell$ with $|\alpha| = c$, and
		$\delta_\alpha \geq 2$ in (\ref{eq:multidegree}),
		and $n_i \geq {\rm max}(2,1+\alpha_i)$
	 	  for $i\in[\ell]$.
Then  the multigraded Hurwitz locus  $\cH^\alpha_X$ is an irreducible variety in $\bG_\alpha$, and
the degree $u$ of the     Hurwitz form ${\rm Hu}^\alpha_X$ satisfies
\begin{equation}
\label{eq:degreeofhurwitz}
        u_i \,\,\le\,\, 2 (g_{\alpha+e_i}+\delta_\alpha - 1) \qquad \hbox{for}  \,\,\, i \in [\ell].
\end{equation}
All inequalities in (\ref{eq:degreeofhurwitz}) are equalities if the variety $X$ is polynodal.
\end{theorem}

\begin{example}[Toric $4$-fold] \label{ex:from88to66}
We revisit our running example,
which is the $4$-fold in $\bP^3 \times \bP^3$ defined by the
$2\times 2$ minors of the matrix (\ref{eq:twobythree}).
From Example \ref{ex:toric4fold} and equation (\ref{eq:genusoneone}) we know
the multidegree 
 $[X] =  2 T_1^2 + 4 T_1 T_2 + 2 T_2^2$
and  the multisectional genus $ g(X)  =   T_1^2 T_2 \,+\, T_1 T_2^2$.
According to Theorem \ref{thm:degree}, the expected degree of the
Hurwitz form ${\rm Hu}_X^{(1,1)}$ equals
\begin{equation}
\label{eq:expected88}
 \bigl( \,2( g_{(2,1)} + \delta_{(1,1)} - 1),  \,2(g_{(1,2)} + \delta_{(1,1)} - 1)\, \bigr) \,\, =\,\, \bigr(\, 8\,,\,8\, \bigr). 
 \end{equation}
However, the degree of  ${\rm Hu}_X^{(1,1)}$ is $(6,6)$, as shown by the explicit formula in the Introduction.

The explanation for this discrepancy is that $X$ fails to be polynodal.
For generic lines $L$ and planes $P$ in $\bP^3$, the curve
$C = X \cap (L \times P)$ is smooth of genus one in $\bP^1 \times \bP^2$.
Its projection into $\bP^2$ is a curve of degree four
which has geometric genus one. It has two singular points, both of which are
cusps. Hence $C$ is not polynodal.
 In the Riemann-Hurwitz formula, we must subtract one
whenever a node is replaced by a cusp. Hence the degree of ${\rm Hu}_X^{(1,1)}$
in the second factor of $\bG_{(1,1)}$ is $6$ instead of $8$. The same
holds for the first factor. 

It is instructive to modify this example by replacing (\ref{eq:twobythree}) with nearby matrices, such as
$$ 
 \begin{bmatrix} \,x_0 y_3 & x_1 x_2 & x_3 -x_2 \,\,\\
\,y_1 y_2 & x_3 y_0 & y_3 \,\,\end{bmatrix}
 \qquad {\rm or} \qquad 
 \begin{bmatrix} \, x_0 y_3 & x_1 x_2 & x_3 -x_2\,\,\\
\,y_1 y_2 & x_3 y_0 & y_3+y_1 \,\,\end{bmatrix}.
$$
Here, $[X]$ and $g(X)$ are unchanged, but
the Hurwitz forms now have degrees $ (7,8)$ and $(8,8)$.
On the left, one projected curve
has one node and one cusp.
On the right, $X$ is polynodal. 
\end{example}

\begin{proof}[Proof of Theorem \ref{thm:degree}]
We first show the irreducibility of 
the multigraded Hurwitz locus~$\cH_X^\alpha$. For every subset $S\subsetneq[\ell]$, 
consider the following Zariski closed subset of the given variety $X$:
\[
    V_S\,\, \coloneqq \,\,\overline{\left\{p\in X_{\rm reg}\ \bigg|\ 
    \text{$\exists L\in\bG_\alpha$ such that $\pi_S(L)\cap\pi_S(X)_{\rm reg}$ is not transverse at $\pi_S(p)$}\right\}}\,.
\]
Let $\mathcal{U}\coloneqq X_{\rm reg}\cap\bigcup_{S\subsetneq[\ell]}(X\setminus V_S)$. 
If $\mathcal{U}=\emptyset$, then $\Phi^\alpha_X=\emptyset$ and there is nothing to prove. 
Otherwise $\mathcal{U}$ is an open dense subset of $X$, since each 
$V_S$ is a closed subvariety. Let $\mathrm{pr}_1$ denote the projection of $X\times\bG_\alpha$ 
onto the first factor $X$.  We have
$\Phi^\alpha_X=\overline{\mathrm{pr}_1^{-1}(\mathcal{U})}$, and all fibers of $\mathrm{pr}_1$ over $\mathcal{U}$ are multilinear of the same dimension. Moreover, the generic fiber is irreducible.
This ensures that $\Phi^\alpha_X$ is irreducible. Since $\cH_X^\alpha$ is the projection of $\Phi^\alpha_X$ to $\bG_\alpha$, we conclude that $\cH_X^\alpha$ is irreducible. Hence 
the Hurwitz form ${\rm Hu}_X^\alpha$ is an irreducible polynomial.

We now prove the statement about the degree of ${\rm Hu}_X^\alpha$.
In the case $\ell=1$ this follows from \cite[Theorem 1.1]{hurwitz}. 
From now on, we assume $\ell\ge 2$.
Fix $i \in [\ell]$ and set $L = L_1 \times \cdots \times L_\ell \subset \bP$,
where the factors are
$L_i \in \bG(\alpha_i+1,n_i)$ and
$L_j \in \bG(\alpha_j,n_j)$ for $j \not= i$. These
linear spaces are assumed to be generic. 
Our hypothesis $\delta_\alpha\ge 2$ ensures
that the following intersection, defined by
 a generic hyperplane $H_i\subset\bP^{n_i}$, is non-empty and
zero-dimensional:
\[
    X \,\cap\, \bigl(L_1\times\cdots\times (L_i\cap H_i)\times \cdots \times L_\ell\bigr)
     \,\,\, =\,\,\,  (X \cap L) \,\cap\, \bigl(\bP^{n_1}\times \cdots\times H_i\times\cdots\times \bP^{n_\ell}\bigr).
\]
 Looking at the right-hand side of the equation above, we conclude that the multidegree $[X\cap L]$ contains the nonzero term $\delta_\alpha\,T_1^{n_1}\cdots T_i^{n_i-1}\cdots T_\ell^{n_\ell}$.
 This means that $X \cap L$ is a curve.

We next argue that the curve $X \cap L$ is irreducible.
Consider the projection $\pi_1$ from $\bP$ onto its first factor $\bP^{n_1}$. Since $X$ is polynodal, 
the curve $X\cap L$ is polynodal. In particular, $\pi_1(X\cap L)$ is a curve.
Furthermore, our hypotheses ensure that $\dim X\ge\ell\ge 2$.

We now show that $\dim\pi_1(X)\ge 2$. Suppose for contradiction that $\dim\pi_1(X)<2$. Then $\dim\pi_1(X\cap L)\le\dim \pi_1(X)\cap L_1\le 0$ because $L_1$ is generic of dimension at most $n_1-1$. This contradicts the fact that $\pi_1(X\cap L)$ is a curve. Applying the version of Bertini's theorem stated in \cite[Th\'{e}or\`{e}me 6.3(4)]{Jouanolou}, we have that $X'\coloneqq\pi_1^{-1}(L_1)\cap X$ is irreducible. We repeat the previous argument, but with $X'$ instead of $X$.
If $\ell=2$ and $X'$ is a curve, then necessarily $L_2=\bP^{n_2}$ and then $X'=X\cap L$, hence $X\cap L$ is irreducible. Otherwise, $\dim X' \ge 2$. Repeating the previous argument with the projection $\pi_2$ from $\bP$ to $\bP^{n_2}$, we obtain 
find that $X''\coloneqq\pi_2^{-1}(L_2)\cap X'=\pi_2^{-1}(L_2)\cap\pi_1^{-1}(L_1)\cap X$ is irreducible.
Repeating the argument inductively, we conclude that $X \cap L$ is 
an irreducible curve.

Let $H$ be a generic hyperplane in $\bP^{n_i}$.
We consider the divisor $H'\coloneqq\bP^{n_1}\times\cdots\times H\times\cdots\times\bP^{n_\ell}$
in $\bP$.
The Hurwitz locus $\mathcal{H}^\alpha_X$ parametrizes subvarieties $L \cap H'$
which intersect $X$ non-transversally, but such that $\pi_S(X) \cap \pi_S(L \cap H')$ is transverse for all $S \subsetneq [\ell]$. This implies that the hyperplane $H$ intersects the projected curve $\pi_i(X \cap L)$  non-transversally. 
The curves $X \cap L$ and  $\pi_i(X \cap L)$ are both irreducible of 
the same degree $\delta_{\alpha}\ge 2$ and the same geometric genus $g_{\alpha+e_i}$.
The dual variety of  $\pi_i(X \cap L)$ is an irreducible hypersurface, defined by an irreducible
polynomial in $H$. The degree of this irreducible polynomial is bounded above by
$2 (g_{\alpha+e_i} + \delta_\alpha -1)$, and equality holds when
all singularities of $\pi(X \cap L)$ are nodes.

In the previous paragraph we assumed that $L$ was generic but fixed.
We now vary $L$ and we view $L \cap H'$ as a point in $\bG_\alpha$.
The multigraded Hurwitz locus $\mathcal{H}^\alpha_X$ is irreducible because its 
restriction to $L$ was seen to be irreducible. Moreover, $\mathcal{H}^\alpha_X$ is a hypersurface
in $\bG_\alpha$, defined by an irreducible polynomial ${\rm Hu}^\alpha_X$.
The degree of this hypersurface in the Pl\"ucker coordinates
of the $i$th factor $\bG(\alpha_i,n_i)$ equals the
degree of the aforementioned dual hypersurface. That degree is bounded
above by $2 (g_{\alpha+e_i} + \delta_\alpha -1)$.
Equality holds for all $i$ if $X$ is polynodal.
\end{proof}

\section{Complete Intersections and Software}
\label{sec4}

In this section we study a generic complete intersection $X$ in
$\bP = \bP^{n_1} \times \cdots \times \bP^{n_\ell}$. We shall derive an explicit 
degree formula for the Hurwitz forms ${\rm Hu}^\alpha_X$.
This rests on
combinatorial expressions for the multidegree $[X]$ and the multisectional genus $g(X)$.
In the second half of this section we present an
implementation in {\tt Macaulay2},
called  {\tt MultiHurwitz.m2}, for computing the numerical invariants studied in this paper.
Our code is posted on {\tt Zenodo}~\cite{zenodo}.

We write the degrees of the polynomials that generate $I_X$ in a matrix $B = (b_{ij}) \in \bN^{c \times \ell}$.
This means that $X$ is the variety defined by $c$ generic multihomogeneous polynomials,
where the $i$th polynomial has degree $b_{ij}$ in the coordinates of the
$j$th projective space $\bP^{n_j}$. Equivalently, the $i$th divisor class equals
$D_i = b_{i1} T_1 + \cdots + b_{i\ell} T_\ell$.
The multidegree of $X$ is
\begin{equation}
\label{eq:mdCI} \quad
[X] \,\, = \sum_{|\alpha|=c} \delta_\alpha\,T^{\alpha} \,\,\,\, = \,\,\,\, D_1 D_2 \cdots D_c \,\, = \,\,
\prod_{i=1}^c \sum_{j=1}^\ell b_{ij} T_j \,\quad \in \,\,\, A^*(\bP).
\end{equation}
What follows  is our main theoretical result in Section \ref{sec4}.
We write $\hat\delta_{ij}$ for Kronecker's delta.

\begin{theorem} \label{thm:CIformula}
For any vector $\alpha \in \bN^\ell$ with $|\alpha| = c$ and $\delta_\alpha \geq 2$,
the multigraded Hurwitz form ${\rm Hu}^\alpha_X$ of the generic complete intersection $X$
is a polynomial of degree $(u_1,\ldots,u_\ell)$, where
\begin{equation}
\label{eq:CIui}
     \quad u_i \,\,\,= \,\,\,2 \delta_\alpha  \, \, +\!
     \sum_{\substack{j\in[\ell]\\\alpha_j+\hat{\delta}_{ij}>0}} 
     \! \delta_{\alpha+e_i-e_j} \cdot \biggl(-\alpha_j-1-\hat{\delta}_{ij}+\sum_{p=1}^c b_{pj} \biggr)\,
\quad \hbox{for all}\,\,\, \,i \in [\ell] .
\end{equation}
\end{theorem}

\begin{proof}
The generic complete intersection $X$ is polynodal. Indeed, 
consider any  generic curve section $C = X \cap (L_1 \times \cdots \times L_\ell)$.
Then $C$ is polynodal becauses its projections into $\prod_{i \in S} \bP^{n_i}$ for $S \subseteq [\ell]$
are curves with at most nodal singularities. Therefore,  by Theorem \ref{thm:degree}, the Hurwitz form
${\rm Hu}_X^\alpha$ has degree $2 (g_{\alpha+e_i}+\delta_\alpha - 1) $
in the Pl\"ucker coordinates of the $i$th Grassmannian
$\bG(\alpha_i,n_i)$. We must show that
$2 (g_{\alpha+e_i}+\delta_\alpha - 1) $ equals
the expression for $u_i$ given in (\ref{eq:CIui}).

We shall use the adjunction formula.
Consider a curve in $\bP$ that is the intersection of
$N-1$ generic divisors of degrees $D_1,D_2,\ldots,D_{N-1}$.
The genus $g$ of this curve satisfies
\begin{equation}
\label{eq:2g-2}
  2g -2 \,\,= \,\,D_1 D_2 \cdots D_{N-1} \cdot (D_1+\cdots+D_{N-1}+K_{\bP})\,. 
  \end{equation}
Here $K_\bP$ denotes the canonical divisor of $\bP $, which is known to be
$$ K_{\bP} \,\, = \,\, \sum_{i=1}^\ell (-n_i-1) \cdot T_i . $$
 Let $D_1, \, \ldots, \, D_c$ be the divisors that define $X$, and 
 let $D_{c+1},\ldots,D_{N-1}$ be the divisors that define
 the multilinear variety $L_1 \times\cdots \times L_\ell$ which intersects $X$
  in a curve of genus $g_{\alpha+e_i}$.
  Let $n$ denote the vector $(n_1, \, \ldots, \, n_\ell)$.
  Then $D_{c+1} \cdots D_{N-1} = T^{n-\alpha}/T_i$, and therefore
      \begin{equation}\label{eq: product Di}
        D_1 \cdots D_c D_{c+1}\cdots D_{N-1} \,\,=\,\,         [X] \cdot D_{c+1} \cdots D_{N-1} \,\,
         = \sum_{\substack{j\in[\ell]\\\alpha_j+\hat{\delta}_{ij}>0}} \delta_{\alpha + e_i - e_j}T^{n - e_j}.
    \end{equation}
The rightmost factor in (\ref{eq:2g-2}) is the divisor class
\begin{equation}\label{eq: sum Di minus K}
  D_1+\cdots+D_{N-1}+K_{\bP}\, \,\,= \,\,\sum_{k=1}^\ell\left(-\alpha_k-\hat{\delta}_{ik}-1+\sum_{p=1}^c b_{pk}\right)T_k.
 \end{equation}
    We multiply \eqref{eq: product Di} with \eqref{eq: sum Di minus K}. Because we are working in the Chow ring $A^*(\bP),$ the $j$th and $k$th terms in each sum multiply to zero unless $j=k.$ The 
    resulting simplified expression is
    $$
        2g_{\alpha + e_i} -2 \,\,= \,\,
        \sum_{j=1}^\ell \delta_{\alpha + e_i - e_j} \biggl(- \alpha_j - 1 - \hat{\delta}_{ij} + \sum_{p=1}^c b_{pj} \biggr).
$$
    By adding $2 \delta_\alpha$ to this expression, we obtain the
    formula for $u_i$ that is asserted in (\ref{eq:CIui}).
\end{proof}

\begin{example}[$\ell = n_1 = n_2 = 2$]
Let $X$ be the surface in $\bP^2 \times \bP^2$ defined by two generic polynomials of
degrees $(b_{11},b_{12})$ and $(b_{21},b_{22})$. By (\ref{eq:mdCI}), the multidegree of $X$ is
   \[ [X] \,\,=\,\, b_{11}b_{21}\,T_1^2+(b_{11}b_{22}\,+\,b_{12}b_{21})T_1T_2\,+\,b_{12}b_{22}\,T_2^2\,.   \]
    Fix the vector $\alpha=(1,1)$. The degree $u=(u_1,u_2)$ of the Hurwitz form ${\rm Hu}^\alpha_X$ satisfies
        \begin{align*}
        u_1 &\,\,=\,\, \delta_{(1,1)}(-1+b_{11}+b_{21})\,\,+\,\,\delta_{(2,0)}(-2+b_{12}+b_{22})\\
        &\,\,=\,\, (b_{11}b_{22}+b_{12}b_{21})(-1+b_{11}+b_{21})\,+\,b_{11}b_{21}(-2+b_{12}+b_{22}),\\
        u_2 &\,\,=\,\, \delta_{(0,2)}(-2+b_{11}+b_{21})\,\,+\,\,\delta_{(1,1)}(-1+b_{12}+b_{22})\\
        &\,\,= \,\,b_{12}b_{22}(-2+b_{11}+b_{21})\,+\,(b_{11}b_{22}+b_{12}b_{21})(-1+b_{12}+b_{22}).
    \end{align*}
    For example, if $(b_{11},b_{12})=(2,1)$ and $(b_{21},b_{22})=(1,1)$, then 
    ${\rm Hu}^\alpha_X$ has degree $u=(6,4)$.
\end{example}

We implemented the formulas above in a {\tt Macaulay2} script, which is
 made available on {\tt Zenodo} \cite{zenodo}. One starts with the command \
 {\tt load "multiHurwitz.m2"}.
The first function is {\tt multiGenera},
which computes the multisectional genus of a variety $X\subset\bP$. This function can take several types of input. For example, we can specify a surface in $\bP=\bP^2\times\bP^2$ by
\[
    B \,\,=\,\, \begin{bmatrix} b_{11} & b_{12} \\ b_{21} & b_{22} \end{bmatrix} \,\,=\,\,
    \begin{bmatrix}
    2&1\\3&4    
    \end{bmatrix}\,.
\]
On the left in  Figure~\ref{fig:multiGenera1} we show how to get the multisectional genus. 
Here, the output is the polynomial $g(X)=21T_1^2T_2+18T_1T_2^2$.  Alternatively, one 
can specify the vector $\alpha$ and obtain the coefficient $g_\alpha$ from $g(X)$. 
This is shown on the right, where the output is $21$.

\medskip

\begin{figure}[h]
\centering
\begin{minipage}{.48\textwidth}
\begin{tcolorbox}[size=fbox,width=\linewidth,colback=blue!5!white,colframe=blue!75!black]
\begin{Verbatim}[fontsize=\small,commandchars=\\\{\}]
N = \{2,2\};
B = matrix\{\{2,1\},\{3,4\}\};
multiGenera(N,B)
\end{Verbatim}
\end{tcolorbox}
\end{minipage}
\begin{minipage}{.48\textwidth}
\begin{tcolorbox}[size=fbox,width=\linewidth,colback=blue!5!white,colframe=blue!75!black]
\begin{Verbatim}[fontsize=\small,commandchars=\\\{\}]
N = \{2,2\};
B = matrix\{\{2,1\},\{3,4\}\};
multiGenera(N,B,\{2,1\})
\end{Verbatim}
\end{tcolorbox}
\end{minipage}
\caption{Commands for {\tt multiGenera} when $X$ is presented by a degree matrix $B$.}\label{fig:multiGenera1}
\end{figure}

A variant of the input format in Figure~\ref{fig:multiGenera1} is offered for our physics application in 
Section~\ref{sec6}. The command {\tt multiGenera} also works when
 $X$ is not a complete intersection. In this case one
 inputs ideal generators for $I_X$.
 This is shown in Figure~\ref{fig:multiGenera2} for our running example (\ref{eq:twobythree}).

\begin{figure}[h]
\centering
\begin{minipage}{.48\textwidth}
\begin{tcolorbox}[size=fbox,width=\linewidth,colback=blue!5!white,colframe=blue!75!black]
\begin{Verbatim}[fontsize=\small,commandchars=\\\{\}]
R = QQ[x0,x1,x2,x3]**QQ[y0,y1,y2,y3];
I = minors(2,matrix\{\{x0*y3,x1*x2,x3\},
                    \{y1*y2,x3*y0,y3\}\})
multiGenera(\{3,3\},I)
\end{Verbatim}
\end{tcolorbox}
\end{minipage}
\begin{minipage}{.48\textwidth}
\begin{tcolorbox}[size=fbox,width=\linewidth,colback=blue!5!white,colframe=blue!75!black]
\begin{Verbatim}[fontsize=\small,commandchars=\\\{\}]
R = QQ[x0,x1,x2,x3]**QQ[y0,y1,y2,y3];
I = minors(2,matrix\{\{x0*y3,x1*x2,x3\},
                    \{y1*y2,x3*y0,y3\}\})
multiGenera(\{3,3\},I,\{1,2\})
\end{Verbatim}
\end{tcolorbox}
\end{minipage}
\caption{Commands for {\tt multiGenera} when $X$ is presented by its ideal $I_X$.}\label{fig:multiGenera2}
\end{figure}

\medskip

The second main function of {\tt multiHurwitz.m2} is called {\tt multidegHurwitz}. This function also accepts several input formats, depending on how the variety $X$ is presented.

\begin{figure}[h]
\centering
\begin{minipage}{.48\textwidth}
\begin{tcolorbox}[size=fbox,width=\linewidth,colback=blue!5!white,colframe=blue!75!black]
\begin{Verbatim}[fontsize=\small,commandchars=\\\{\}]
B = matrix\{\{2,1\},\{3,4\}\};
multidegHurwitz(\{2,2\},B,\{1,1\})
netList multidegHurwitz(\{2,2\},B)
\end{Verbatim}
\end{tcolorbox}
\end{minipage}
\begin{minipage}{.48\textwidth}
\begin{tcolorbox}[size=fbox,width=\linewidth,colback=blue!5!white,colframe=blue!75!black]
\begin{Verbatim}[fontsize=\small,commandchars=\\\{\}]
R = QQ[x0,x1,x2,x3]**QQ[y0,y1,y2,y3];
I = minors(2,matrix\{\{x0*y3,x1*x2,x3\},
                    \{y1*y2,x3*y0,y3\}\})
netList multidegHurwitz(\{3,3\},I)
\end{Verbatim}
\end{tcolorbox}
\end{minipage}
\caption{The function {\tt multidegHurwitz} for the two input formats
in Figures \ref{fig:multiGenera1} and  \ref{fig:multiGenera2}.
\label{fig:multiGenera3}}
\end{figure}

\noindent The  {\tt netList} output  in Figure \ref{fig:multiGenera3}
has one row per term in the multidegree.
On the left, we get $[X] = 6 T_1^2 + 11 T_1 T_2 + 4 T_2^2$.
For the middle term, we learn that ${\rm Hu}^{(1,1)}_X$ has degree $(62,56)$.
On the right, we obtain the prediction $(8,8)$ 
from (\ref{eq:expected88}) for the degree of
${\rm Hu}^{(1,1)}_X$.  But, the two expected nodes are
actually cusps, as explained in Example \ref{ex:from88to66}.
 The true degree is~$(6,6)$.

\section{Polytopes and Nash Equilibria}
\label{sec5}

This section is motivated by game theory, namely by the
study of totally mixed Nash equilibria.
These are defined by semi-mixed systems of multilinear equations
whose parameters are the payoff tables of the players \cite[Chapter 6]{CBMS}.
The Nash discriminant \cite{vectorBundleNash} is a polynomial
in the payoff tables which vanishes when two equilibria come together.
We will show that these can be interpreted as
multigraded Hurwitz forms for a certain class of toric varieties.

Before getting to Nash equilibria, we take a step back
and revisit the mixed discriminants in Examples \ref{ex:md} and \ref{ex:semimd}.
Fix $d \geq \ell$ and subsets $A_1,\ldots,A_\ell$ that span $\bZ^d$, and set
$P_i = {\rm conv}(A_i)$ for $i \in [\ell]$. Let $X $ be  the
associated $d$-dimensional toric variety in $ \bP = \bP^{n_1} \times \cdots \times \bP^{n_\ell} $.
Here  we abbreviate $n_i = |A_i| - 1$ for $i \in [\ell]$.
The codimension  of $X$ is $c = n_1+\cdots+n_\ell - d$.

With the tuple of polytopes $P_1,\ldots,P_\ell$ one associates the following {\em volume polynomial}
\begin{equation}
\label{eq:volumepolynomial}
V(T) \,\,=\,\, {\rm vol} \bigl(T_1 P_1 + T_2 P_2 + \cdots + T_\ell P_\ell \bigr)  \,\,\, = \,\,\, 
\sum_{|\gamma| = d}  \mu_\gamma \,T^\gamma.
\end{equation}
Here ``vol'' is the Lebesgue measure on $\bR^d$.
It is known that (\ref{eq:volumepolynomial}) is a Lorentzian polynomial,
which means that it has desirable combinatorial properties \cite{Huh}.
According to the following formula from \cite[Section 2.1]{Huh},
the coefficients of the volume polynomial are mixed volumes:
$$ MV ( P_{i_1},P_{i_2},\ldots, P_{i_d}) \,\, = \,\, \frac{\partial}{\partial T_{i_1}} 
 \frac{\partial}{\partial T_{i_2}}  \cdots   \frac{\partial}{\partial T_{i_d}}  V(T) \qquad
 \hbox{for any} \,\, i_1,i_2,\ldots,i_d \,\in\,  [\ell].
 $$
By Bernstein's Theorem, this mixed volume is the number of
solutions to a square polynomial system with Newton polytopes
$P_{i_1},P_{i_2},\ldots, P_{i_d} $. And, that number is precisely
 the coefficient complementary to $ T_{i_1} T_{i_2} \cdots T_{i_d}$ in the
multidegree of the toric variety $X$ associated with the subsets $A_1,A_2,\ldots,A_\ell$ of $\bZ^d$.
These considerations directly imply the following result.

\begin{lemma} \label{lem:MV}
The multidegree of the toric variety $X \subset \bP$ is read off from (\ref{eq:volumepolynomial}) as follows:
$$ [X] \,\,\, = \,\,\, 
\sum_{|\gamma| = d}  \mu_\gamma \,\gamma_1! \,\gamma_2! \,\cdots\, \gamma_\ell ! \cdot
T_1^{n_1-\gamma_1} T_2^{n_2-\gamma_2}\,\cdots\, T_\ell^{n_\ell-\gamma_\ell} .
$$
Hence, in the notation of (\ref{eq:multidegree}), the coefficients in the multidegree of $X$ are
$$ \delta_\alpha \,\, = \,\, \mu_{n_1-\alpha_1,\ldots,n_\ell-\alpha_\ell} \cdot \prod_{i=1}^\ell (n_i - \alpha_i)!. $$
\end{lemma}

We next determine
the multisectional genus $g(X)$. This rests on Khovanskii's theorem~\cite[Theorem 1]{Kho}. 
Namely, we compute each coefficient $g_\beta$
by counting lattice points in Newton polytopes.
To state the theorem, we fix $\beta \in \bN^\ell$ with $|\beta| = c+1$.
We need to find the genus $g_\beta$ of the curve
$\,C  = X \,\cap \, (L_1 \times \cdots \times L_\ell) $,
where $L_i $ is generic in $ \bG(\beta_i,n_i)$ for $i \in [\ell]$.
The subspace $L_i$ is defined by $n_i-\beta_i$ linear equations
on the $i$th factor $\bP^{n_i}$, and these pull back to $n_i-\beta_i$ 
Laurent polynomial equations in $d$ variables with Newton polytope $P_i$.
For any vector $\mu \in \bN^\ell$ with $\mu \leq \beta$
in the coordinatewise order, we consider the Minkowski sum
$$P_\mu \, = \,\mu_1 P_1 + \mu_2 P_2 + \cdots + \mu_\ell P_\ell
\qquad \subset \,\, \bR^d. $$
We compute (\ref{eq:msg}) by counting
the interior lattice points in each of these polytopes.

\begin{theorem}[Khovanskii] \label{thm:khovanskii} The
coefficients of the multisectional genus $g(X)$  are
$$ g_\beta \,\,\, = \,\,\,\sum_{\gamma \leq \beta} \,(-1)^{|\beta-\gamma|}  \cdot
 \bigl|\, {\rm int}(P_\gamma) \,\cap \,\bZ^d \,\bigr|. $$
\end{theorem}

We illustrate Lemma \ref{lem:MV} and Theorem \ref{thm:khovanskii} for our running example.

\begin{example}[Toric $4$-fold]
Here $d = 4$, $\ell = 2$, and the supports are $A_1 = \{e_1+e_2,e_3,e_4,0\}$ and $A_2 = \{e_3+e_4,e_1,e_2,0\}$.
So, $P_1,P_2$ are tetrahedra in $\bR^4$. The volume polynomial equals
$$ V(T) \,\,=\,\, {\rm vol}(T_1 P_1 + T_2 P_2) \,\,\,=\,\,\, \frac{1}{3} T_1^3 T_2 \,+\,  T_1^2 T_2^2 \,+\, \frac{1}{3} T_1 T_2^3. $$
Lemma \ref{lem:MV} yields $[X] = 2 T_1^2 + 4 T_1 T_2  + 2T_2^2$. 
Theorem \ref{thm:khovanskii} implies that
the multisectional genus is $g(X) = T_1^2 T_2 + T_1 T_2^2$.
Indeed, the $4$-polytope
$P_{(2,1)} = 2P_1 +P_2$ has  precisely one interior lattice point, while none
of the polytopes $P_i+P_j$ has an interior lattice point.
Hence the system of  equations with Newton polytopes $P_1,P_1,P_2$
defines a curve of genus one. 
\end{example}

We now turn to game theory, following \cite{vectorBundleNash}.
First consider $3$ players with binary~choices.

\begin{example} \label{ex:222game}
 The {\em Nash equilibrium variety} of a $2\times 2\times 2$ game is the locus of points $([x_0:x_1],[y_0:y_1],[z_0:z_1])\in\bP^1\times\bP^1\times\bP^1$ that satisfy the three bilinear equations
\begin{align}\label{eq: tmNe 2x2x2 game}
    \begin{split}
        a_{00}y_0z_0+a_{01}y_0z_1+a_{10}y_1z_0+a_{11}y_1z_1 &= 0,\\
        b_{00}x_0z_0+b_{01}x_0z_1+b_{10}x_1z_0+b_{11}x_1z_1 &= 0,\\
        c_{00}x_0y_0+c_{01}x_0y_1+c_{10}x_1y_0+c_{11}x_1y_1 &= 0.
    \end{split}
\end{align}
This system has two solutions when the matrices $A=(a_{ij})$, $B=(b_{ij})$, $C=(c_{ij})$ are generic.
The Nash discriminant, which identifies double roots,
is the symmetric $6 \times 6$ determinant
\begin{equation}
\label{eq:nashdiscriminant}
    {\rm Hu}_X^{(2,2,2)} \,\, = \,\,
    {\rm det}\begin{small}
    \begin{bmatrix}
     \,   0 & A & B \, \\
     \,   A^T & 0 & C \,\\
      \,  B^T & C^T & 0\,
    \end{bmatrix}\end{small}.
\end{equation}
The toric variety $X$ is
the mixed Segre embedding of $\bP^1\!\times\!\bP^1\!\times\!\bP^1$ into 
$\bP^3 \!\times\! \bP^3\!\times \! \bP^3$, given by
\[
    \bigl([x_0:x_1],[y_0:y_1],[z_0:z_1] \bigr)\,\,\mapsto \,\,
    \bigl([y_0z_0:\cdots:y_1z_1],[x_0z_0:\cdots:x_1z_1],[x_0y_0:\cdots:x_1y_1] \bigr)\,.
\]
Here $d=\ell=3$, $\,P_1,P_2,P_3$ are facets of the cube $[0,1]^3$, and
$V(T) = (T_1+T_2)(T_1+T_3)(T_2+T_3)$.
By applying Lemma \ref{lem:MV}, we find that the multidegree of the toric threefold $X$ equals
\[
    [X] \,\, =\,\, 2\,T_{1}^{3}T_{2}^{2}T_{3}+2\,T_{1}^{3}T_{2}T_{3}^{2}+2\,T_{1}^{2}T_{2}^{3}T_{3}
    + {\bf 2}\,T_{1}^{2}T_{2}^{2}T_{3}^{2}+2\,T_{1}^{2}T_{2}T_{3}^{3}+2\,T_{1}T_{2}^{3}T_{3}^{2}+2\,T_{1}T_{2}^{2}T_{3}^{3}.
    \]
    The Hurwitz form in (\ref{eq:nashdiscriminant}) has degree $(2,2,2)$ because
the multisectional genus equals
 $$ g(X) \,\,=\,\, -1 \cdot (T_{1}^{3}T_{2}^{3}T_{3} + T_{1}^{3}T_{2}T_{3}^3 + T_{1}T_{2}^3 T_{3}^3 ) \,+\,
 0 \cdot (
T_{1}^{3}T_{2}^{2}T_{3}^{2}+T_{1}^{2}T_{2}^{3}T_{3}^{2}+T_{1}^{2}T_{2}^{2}T_{3}^{3}). $$
The coefficient $0$ of $T_{1}^{3}T_{2}^{2}T_{3}^{2}$ arises because $P_2+P_3$ is
$3$-dimensional but has no interior lattice point. The coefficient $-1$ in front of
$T_{1}^{3}T_{2}^{3}T_{3}$ indicates that $P_3+P_3$ is not full-dimensional.
\end{example}

We now turn to games with $\ell$ players of format
$(k_1+1)\times\cdots\times(k_\ell+1)$. We assume $1\le k_1\le\cdots\le k_\ell$. 
The Nash discriminant variety was defined in \cite[Section 3]{vectorBundleNash}.
It has codimension one if $k_\ell\le\sum_{i=1}^{\ell-1}k_i$ by \cite[Theorem 3.7]{vectorBundleNash}, but there may be
extraneous components of lower dimensions.
This was shown in \cite[Example 3.27]{vectorBundleNash} for 
games of boundary format $2 \times 2 \times 3 $.

To model the general case using polytopes, we set $d = k_1 + k_2 + \cdots + k_\ell$,
and we consider the product of simplices $\Delta := \Delta_{k_1} \times \cdots \times \Delta_{k_\ell} $,
where $\Delta_{k_i} = {\rm conv}\{0,e_1,\ldots,e_{k_i} \}$ denotes the unit simplex.
We fix the following $\ell$ distinguished facets of the $d$-dimensional polytope $\Delta$:
$$ A_i \,\, := \,\, \Delta_{k_1} \times  \cdots \times \Delta_{k_{i-1}} \times \{0\} \times
\Delta_{k_{i+1}} \times \cdots \times \Delta_{k_\ell} \qquad {\rm for} \,\,\, i \in [\ell]. $$
See \cite[Section 6.4]{CBMS}. These  polytopes define a toric variety  $X $
in $\bP = \bP^{n_1} \times  \cdots \times \bP^{n_\ell}$, where
$$\, n_i \,\,= \,\,(k_1+1)\cdots (k_{i-1}+1)(k_{i+1}+1) \cdots (k_\ell+1)\,-\,1 
\qquad {\rm for} \,\,\, i \in [\ell]. $$
Just like in Example \ref{ex:222game}, the variety $X$ is an embedding of the
Segre product $\bP^{k_1} \times \cdots \times \bP^{k_\ell}$.
The volume polynomial equals
$V(T) =  \prod_{i=1}^\ell ( \sum_{j \not= i} T_j)^{k_i}$,
and from this we can read off the multidegree $[X]$.
Set $\alpha = (n_1-k_1,n_2-k_2,\ldots,n_\ell-k_\ell)$.
The coefficient $\delta_\alpha$ in $[X]$ is the 
number of totally mixed Nash equilibria.
A combinatorial rule for $\delta_\alpha$ is given in \cite[Theorem 6.8]{CBMS}.
We seek conditions on the payoff tables
that force two equilibria to come together.

\smallskip

In the following result, the upper bound of Theorem~\ref{thm:degree} 
is made more explicit for 
Nash equilibria.
Our formula is computed in the Chow ring of $\Sigma\coloneqq\bP^{k_1} \times \cdots \times \bP^{k_\ell}$, that is
\[
    A^*(\Sigma) \,\,=\,\, \bZ[H_1,H_2,\ldots,H_\ell] / \bigl\langle \,
    H_1^{k_1+1},  \,H_2^{k_2+1},\,\ldots, \,H_\ell^{k_\ell+1} \bigr\rangle\,.
\]
As is customary in intersection theory, 
the integral sign $\int_\Sigma$ will be used to mean that we extract the coefficient of $H_1^{k_1} H_2^{k_2} \cdots H_\ell^{k_\ell}$
  in the argument. In particular, we have the identity
\begin{equation}\label{eq:delta}
	\delta_\alpha \,\,= \,\,\int_\Sigma H\,,\quad
	\text{where}\,\,\, H\coloneqq \prod_{i=1}^\ell\hat{H}_i^{k_i}\ \,\,
	\text{and}\,\,\, \hat{H}_i\coloneqq\sum_{j\neq i}H_j\,\,\,\text{for all $i\in[\ell]$.}
\end{equation}
Here, the relevant degree is  $\alpha = (n_1-k_1,n_2-k_2,\ldots,n_\ell-k_\ell)$.

\begin{proposition}\label{prop: multidegree games}
If $\,\mathrm{Hu}_X^{\alpha}$ is a hypersurface, then its degree $(u_1,\ldots,u_\ell)$ satisfies
\begin{equation}
\label{eq:nashbound}
	u_i \,\,\,\le\,\, 
\int_\Sigma \, \, 2\,H+\frac{H}{\hat{H}_i}\left(-\hat{H}_i+\sum_{j=1}^\ell \left[k_j\hat{H}_j-(k_j+1)H_j\right]\right)\,.
\end{equation}
The upper bound in (\ref{eq:nashbound}) is an equality provided
the toric variety $X$ is polynodal.
\end{proposition}

We conjecture that $X$ is polynodal, and $\mathrm{Hu}_X^{\alpha}$ is a hypersurface, so
equality holds in (\ref{eq:nashbound}).

\begin{proof}
	Assume that $\mathrm{Hu}_X^{\alpha}$ is a hypersurface. By Theorem~\ref{thm:degree}, its degree 
	$(u_1,\ldots,u_\ell)$ satisfies
	\begin{equation}\label{eq: upper bound}
	u_i \,\,\le \,\,2(\delta_\alpha+g_{\alpha+e_i}-1)\,.
	\end{equation}
	The coefficient $\delta_\alpha$ in $[X]$ satisfies \eqref{eq:delta}.
    To determine the upper bound explicitly, it remains to compute the coefficient $g_{\alpha+e_i}$ in the multisectional genus $g(X)$ for all $i\in[\ell]$. This is the genus of the curve $X\cap (L_1\times\cdots\times L_\ell)$, where $\dim(L_j)=\alpha_j+\hat{\delta}_{ij}$. Viewing $X$ as a Segre embedding of $\Sigma$, we pull back $L_1 \times \cdots
    \times L_\ell $ to obtain the generic complete intersection $Z_1\times\cdots\times Z_\ell$ in $\Sigma$. Each $Z_j$ is cut out by $k_j-\hat{\delta}_{ij}$ generic multihomogeneous polynomials of degrees $(1,\ldots,1)-e_j$ in the $\ell$ sets of variables. Overall we are considering a generic complete intersection of $k-1\coloneqq k_1+\cdots+k_\ell-1$ divisors in $\Sigma$. We denote these by
     $D_1,\ldots,D_{k-1}$. The genus $g_{\alpha+e_i}$
          can thus be derived using the adjunction formula as in Theorem \ref{thm:CIformula}, namely
	\begin{equation}\label{eq:genus formula complete intersection}
	2\,g_{\alpha+e_i}-2 \,\,= \,\,\int_\Sigma D_1\cdots D_{k-1}\left(D_1+\cdots+D_{k-1}+K_{\Sigma}\right).
	\end{equation}
    Here $K_\Sigma$ is the canonical divisor of $\Sigma$.
    We next substitute the following relations into
		 \eqref{eq:genus formula complete intersection}:
	\begin{align*}
		D_1\cdots D_{k-1} \,=\, \frac{H}{\hat{H}_i}\,,\quad D_1+\cdots+D_{k-1} \,=\, -\hat{H}_i+\sum_{j=1}^\ell k_j\hat{H}_j\,,\quad K_\Sigma \,=\, -\sum_{j=1}^\ell(k_j+1)H_j.
	\end{align*}
By applying \eqref{eq: upper bound}, we obtain the upper bound stated in (\ref{eq:nashbound}).
\end{proof}

\begin{example}[Binary games]
Consider a game with $\ell$ players with binary choices. Here,
$k_1=\cdots=k_\ell=1$ and $\,n_1 =  \cdots = n_\ell = 2^{\ell-1}-1$. The degree
  $\delta_\alpha$  is the number of derangements of 
  the set $[\ell]$. See 
  \cite[Remark 2.9]{vectorBundleNash}  and  \cite[Corollary 6.9]{CBMS}.
  The Nash discriminant is the Hurwitz form ${\rm Hu}^\alpha_X$
  for $\alpha=(2^{\ell-1}-2,\ldots,2^{\ell-1}-2)$.
  Since $d = \ell$, this is an instance of the fully mixed discriminant in
Example~\ref{ex:md}.
The upper bound in Proposition~\ref{prop: multidegree games} simplifies to
\[
	u_i \,\,\le \,\,(\ell-1)!\,\sum_{j=0}^\ell\frac{(-1)^j}{j!}\left[(\ell-3)(\ell-j)+\ell\right]\,.
\]
Consider the case $\ell = 4$, 
where $X \simeq (\bP^1)^4 $ is embedded in $\bP = (\bP^7)^4$,
and  $\alpha = (6,6,6,6)$. We verified that the multidegree of ${\rm Hu}^\alpha_X$
is $(20,20,20,20)$. The expected degrees $u_i$ for $\ell \in \{3, \ldots, 10\}$ are respectively $2, 20, 150, 1192, 10330, 98268, 1023470, 11614160$.
\end{example}

\begin{example}[Games of format $3 \!\times\! 3\! \times \!3$]
The toric variety $X \simeq (\bP^2)^3$ lives in $(\bP^8)^3$.
Its ideal $I_X$ is generated by $108$ quadratic binomials
in $27$ variables.
The multidegree has $19$~terms:
$$ [X] \,= \, 6 \,T_1^8 T_2^6 T_3^4 \,+ \,9 \,T_1^8 T_2^5 T_3^5 \,+\, \cdots\, +\, {\bf 10} \,T_1^6 T_2^6 T_3^6 
\,+\, \cdots\, + \,6 T_1^4 T_2^7 T_3^7  \,+\,6 T_1^4 T_2^6 T_2^8 . $$
The relevant degree is $\alpha = (6,6,6)$, so generically there are $10$ totally mixed Nash equilibria. Applying Proposition~\ref{prop: multidegree games}, the expected multidegree of ${\rm Hu}^\alpha_X$ is $(24,24,24)$. Here $24$ is the degree in the $28$ Pl\"ucker coordinates of
the Grassmannian $\bG(6,8)$, one for each of the three players.
The total degree in the entries of the three payoff tables equals
$48+48+48=144$.
\end{example}

\section{Grassmannians and Particle Physics}
\label{sec6}

In this section we discuss an application of multigraded Hurwitz forms
to Landau analysis for Feynman integrals in particle physics \cite{HMPS1, HMPS2}.
We begin by describing the general setting.
We fix $ n_1 = n_2 = \cdots = n_\ell = 5$, so our ambient product of projective spaces is
$\bP = (\bP^5)^\ell$. The  six  coordinates on the $i$th factor $\bP^5$ are
denoted  $x^{(i)} = ( x^{(i)}_{12}, x^{(i)}_{13}, x^{(i)}_{14} , x^{(i)}_{23}, x^{(i)}_{24}, x^{(i)}_{34})$.
In each factor we consider the Grassmannian of lines in $\bP^3$ in its Pl\"ucker embedding. We obtain
\begin{equation}
\label{eq:Grell}
{\bG}(1,3)^\ell \,\,=\,\, \bigl\{\,
(x^{(1)},\,x^{(2)},\,\ldots, \,x^{(\ell)}) \in \bP  \,:\,
x^{(i)}_{12} x^{(i)}_{34} - x^{(i)}_{13} x^{(i)}_{24} + x^{(i)}_{14} x^{(i)}_{23}  \,=\, 0
\,\,\,\,\hbox{for}\,\, i \in [\ell] \,\bigr\}.
\end{equation}
This is a complete intersection
with degree vectors $2 e_1, 2 e_2, \ldots , 2 e_\ell$, so its
multidegree equals
$$ \bigl[\bG(1,3)^\ell \bigr] \,\,= \,\, 2^\ell\, T_1 T_2 \cdots T_\ell. $$
We shall view $\bG(1,3)^\ell$ as a generic complete intersection, given by a quadric in each~factor.

\smallskip

Our next ingredient is a graph $G$ with vertex set $[\ell] = \{1,2,\ldots,\ell\}$.
We identify $G$ with its set of edges, which is a subset of $\binom{[\ell]}{2}$.
For each edge $ij$ of $G$ we consider the bilinear form
\begin{equation}
\label{eq:incidencerel}
B_{ij}(x) \,\,:= \,\,
 x^{(i)}_{12} x^{(j)}_{34} 
- x^{(i)}_{13} x^{(j)}_{24}
+ x^{(i)}_{14} x^{(j)}_{23} 
+ x^{(i)}_{23} x^{(j)}_{14}
- x^{(i)}_{24} x^{(j)}_{13} 
+ x^{(i)}_{34} x^{(j)}_{12} .
 \end{equation}
 The equation $B_{ij}(x) = 0$
 holds if and only if the $i$th line intersects the $j$th line in $\bP^3$.
 
 The graph $G$ specifies the incidence variety 
 \begin{equation}
V_G \,\, = \,\, \bigl\{ \,x \in \bG(1,3)^\ell \,: \, B_{ij}(x) = 0 \,\,\, \,{\rm for} \,\,\, ij \in G \, \bigr\}.
\end{equation}
This variety was introduced in~\cite{HMPS1} and  is applied to 
Feynman integrals in \cite{HMPS2}. Its points are 
configurations of $\ell$ lines in $\bP^3$ which are required to satisfy
pairwise incidence relations.

Objects of main interest in \cite{HMPS2} are the {\em LS discriminants} of the graph $G$.
The acronym LS refers to the {\em leading singularities} of the integrand in a Feynman integral.
Analyzing such singularities is a key point of Landau analysis.
The connection to the present work arises because the LS discriminants are the multigraded Hurwitz forms of the incidence variety $V_G$.

\begin{example}[$\ell=1$] \label{ex:oldschubert}
Here $G$ is the graph with one vertex and no edge,
and $V_G$ is simply the Grassmannian $\bG(1,3)$ in $\bP^5$.
We are thus in the classical setting of \cite{hurwitz},
with only one projective space. We have $c=1$,
the sectional genus is $g=0$, and we set $x_{jk} = x^{(1)}_{jk}$.
The Hurwitz form of $\bG(1,3)$ is a hypersurface in $\bG(1,5)$.
Its points are the lines $\mathfrak{L}$ in $\bP^5$ for which the intersection  $\bG(1,3) \cap \mathfrak{L}$ is
a double point.  We represent $\mathfrak{L}$ as the row space of a matrix
\begin{equation}
\label{eq:twobysix}
\begin{bmatrix}
u_{12} & u_{13} & u_{14} & u_{23} & u_{24} & u_{34} \\
v_{12} & v_{13} & v_{14} & v_{23} & v_{24} & v_{34} 
\end{bmatrix}.
\end{equation}
Writing $[ij]$ for the $2 \times 2$ minors of the matrix in (\ref{eq:twobysix}), the Hurwitz form  of $\bG(1,3)$ equals
$$ {\rm Hu}_{\bG(1,3)}^{(4)} \,=\,
 [16]^2 + [25]^2 + [34]^2 +2 [16] [25]-2 [16] [34]
+2 [25] [34]+4 [14] [36]-4 [15] [26]-4 [24] [35]. $$
In our physics application, the line $\mathfrak{L}$ is the intersection of four Schubert hyperplanes in $\bP^5$:
\begin{equation}
\label{eq:fourbysix}
\begin{matrix} 
a^{(1)}_{34} x_{12} -  a^{(1)}_{24} x_{13} + a^{(1)}_{23} x_{14} + a^{(1)}_{14} x_{23} - a^{(1)}_{13} x_{24} + a^{(1)}_{12} x_{34} & = & 0,
\smallskip \\
a^{(2)}_{34} x_{12} -  a^{(2)}_{24} x_{13} + a^{(2)}_{23} x_{14} + a^{(2)}_{14} x_{23} - a^{(2)}_{13} x_{24} + a^{(2)}_{12} x_{34} & = & 0,
\smallskip \\
a^{(3)}_{34} x_{12} -  a^{(3)}_{24} x_{13} + a^{(3)}_{23} x_{14} + a^{(3)}_{14} x_{23} - a^{(3)}_{13} x_{24} + a^{(3)}_{12} x_{34} & = & 0,
\smallskip \\
a^{(4)}_{34} x_{12} -  a^{(4)}_{24} x_{13} + a^{(4)}_{23} x_{14} + a^{(4)}_{14} x_{23} - a^{(4)}_{13} x_{24} + a^{(4)}_{12} x_{34} & = & 0.
\end{matrix}
\end{equation}
The $i$th hyperplane being Schubert means that its coefficient vector
$( a^{(i)}_{12},  a^{(i)}_{13}, \ldots, a^{(i)}_{34})$ lies in the Grassmannian $\bG(1,3)$.
We write $ {\rm Hu}_{\bG(1,3)}^{1} $ in the
$24$ Pl\"ucker coordinates $a^{(i)}_{jk}$ by replacing
$[ij]$ with the signed $4 \times 4$ minor of the coefficient matrix in (\ref{eq:fourbysix}) indexed by $[6] \backslash \{i,j\}$.
What results is the most basic discriminant in Schubert calculus.
Given four lines in $\bP^3$, there are two lines incident to all four.
Our Hurwitz form vanishes when these two lines coincide.
\end{example}

The geometry studied in
 \cite{HMPS1, HMPS2} extends this basic Schubert problem from $\ell = 1$ to $\ell \geq 2$.

\begin{example}[$\ell=2$] \label{ex:oneedge}
Here $G$ is the graph with two vertices and one edge.
The variety $V_G$ is irreducible. It is defined by one bilinear equation
(\ref{eq:incidencerel})  in $\bG(1,3) \times \bG(1,3)$.
Hence $V_G$ is a complete intersection of
type $(2,0),(0,2),(1,1)$ in  $\bP^5 \times \bP^5$.
So, its codimension equals $c=3$.

To proceed, we replace $V_G$ with 
a generic complete intersection 
in $\bP^5 \times \bP^5$, of the same type 
$(2,0),(0,2),(1,1)$, to be denoted $U_G$.
Its multidegree and multisectional genus are
$$ [U_G] \,=\, 4 T_1^2 T_2 \,+\, 4 T_1 T_2^2 \quad {\rm and} \quad
g(U_G) \,=\, T_1^4 - T_1^3 T_2 + T_1^2 T_2^2 - T_1 T_2^3 + T_2^4.
$$
The multigraded Hurwitz form  ${\rm Hu}^{(1,2)}_{U_G}$ is an irreducible polynomial 
of degree $(8,4)$  in the Pl\"ucker coordinates of 
a line $L \in \bG(1,5)$ and a plane $P \in \bG(2,5)$.
The intersection $U_G \,\cap \,(L \times P)$
consists of four points, and ${\rm Hu}^{(1,2)}_{U_G}$ 
vanishes when two of the four points come together.

We obtain the Hurwitz form of our incidence variety $V_G$ by degenerating the
Hurwitz form ${\rm Hu}^{(1,2)}_{U_G}$ of the generic
complete intersection $U_G$. The
three generic polynomials are now replaced by
the two Pl\"ucker quadrics and one incidence relation (\ref{eq:incidencerel}).
The resulting hypersurface in $ \bG(1,5) \times \bG(2,5)$
decomposes into two irreducible factors, of degrees $(4,4)$ and $(4,0)$.
The first is the desired Hurwitz form ${\rm Hu}^{(1,2)}_{V_G}$.
The second is an extraneous factor, namely it is the
Hurwitz form from Example \ref{ex:oldschubert}, taken
in the left factor of $ \bG(1,5) \times \bG(2,5)$.

In conclusion, the variety $V_G$ has two multigraded Hurwitz forms
${\rm Hu}^{(1,2)}_{V_G}$ and ${\rm Hu}^{(2,1)}_{V_G}$. Both of
these are irreducible polynomials of degree $(4,4)$ in pairs of
Pl\"ucker coordinates. For the application of \cite{HMPS2},
these are replaced by expressions in the $42$ coefficients
$a^{(i)}_{jk}$ that specify $7=4+3$ Schubert hyperplanes in $\bP^5$,
four for $\bG(1,5)$, as in (\ref{eq:fourbysix}), and three for  $ \bG(2,5)$.
\end{example}

The incidence variety $V_G$ can be irreducible or reducible,
depending on the graph $G$. A criterion that is necessary and sufficient 
appears in \cite[Theorem 5.8]{HMPS1}. Likewise,
$V_G$ may or may not be a complete intersection.
A characterization is found in \cite[Theorem 5.4]{HMPS1}.
Roughly speaking, $V_G$ is a complete intersection if and only if $G$
satisfies a certain sparsity condition, and this condition is
always satisfied when $G$ represents a  planar Feynman diagram.

In what follows we assume that $V_G$ is a complete intersection.
However, $V_G$ is generally reducible. The codimension
of $V_G$ in $(\bP^5)^\ell$ equals $c = \ell+ |G|$, where $|G|$ is the number of edges in $G$.
Let $B_G$ be the $c \times \ell$ matrix whose rows are the vectors
$2e_i$ for $i \in [\ell]$ and $e_i+e_j$ for $ij \in G$.
Then $V_G$ is a complete intersection of type $B_G$.
We write $U_G$ for the generic complete intersection
of the same type $B_G$ in $(\bP^5)^\ell$. The common multidegree equals
\begin{equation}
\label{eq:VGUG}
 [V_G] \,=\, [U_G] \, \,=\,\,\sum_{|\alpha| = c} \delta_\alpha T^\alpha \,\,\,=\,\,\,
2^\ell\, T_1 T_2 \cdots T_\ell \cdot \prod_{ij \in G} (T_i + T_j) .
\end{equation}
Our approach is to use the results for $U_G$ given in Section \ref{sec4}
to glean information about $V_G$.
Namely, whenever the coefficient $\delta_\alpha$ in the multidegree (\ref{eq:VGUG}) 
is at least $2$, the Hurwitz form ${\rm Hu}^{(\alpha)}_{U_G}\,$ is an irreducible
polynomial in the Pl\"ucker coordinates of $\bG_\alpha$.
Theorem \ref{thm:CIformula} implies:

\begin{corollary}
The multigraded Hurwitz form ${\rm Hu}^{\alpha}_{U_G}$ has degree $u=(u_1,\ldots,u_\ell)$ where
$$
     \quad u_i \,\,\,= \,\,\,2 \delta_\alpha  \, \, +\!
     \sum_{\substack{j\in[\ell]\\\alpha_j+\delta_{ij}>0}}
     \! \delta_{\alpha+e_i-e_j} \cdot \biggl(1-\alpha_j-\hat{\delta}_{ij}+{\rm degree}_G(i)
           \biggr)\,
\quad \hbox{for all}\,\,\, \,i \in [\ell] .
$$
The degree of the Hurwitz form $\,{\rm Hu}^{\alpha}_{V_G}$ 
of the incidence variety $V_G$ is bounded above by $u_i$.
\end{corollary}

We conjecture that the upper bound in the last sentence is an equality
if we take extraneous factors into consideration.
Thus, the computation for $U_G$ also yields
essential information for $V_G$. Our software in \cite{zenodo} allows for a convenient
input format. Namely, we can input the graph $G$ by its list of edges.
Figure \ref{fig:multiGenera4} shows the computation of $g(U_G)$ for
Example~\ref{ex:oneedge}.

\begin{figure}[h]
\centering
\begin{minipage}{.48\textwidth}
\begin{tcolorbox}[size=fbox,width=\linewidth,colback=blue!5!white,colframe=blue!75!black]
\begin{Verbatim}[fontsize=\small,commandchars=\\\{\}]
G = \{\{1,2\}\};
multiGenera(2,G)
netList multidegHurwitz(2,G)
\end{Verbatim}
\end{tcolorbox}
\end{minipage}
\begin{minipage}{.48\textwidth}
\begin{tcolorbox}[size=fbox,width=\linewidth,colback=blue!5!white,colframe=blue!75!black]
\begin{Verbatim}[fontsize=\small,commandchars=\\\{\}]
G = \{\{1,2\}\};
multiGenera(2,G,\{3,1\})
multidegHurwitz(2,G,\{3,1\})
\end{Verbatim}
\end{tcolorbox}
\end{minipage}
\caption{The commands from Figures \ref{fig:multiGenera1} and \ref{fig:multiGenera3}
when  $X = U_G$ for a graph $G$.}
\label{fig:multiGenera4}
\end{figure}

\begin{example}[$\ell=3$] \label{ex:twoedges}
Figure \ref{fig:multidegHurwitz} shows input and output of {\tt multidegHurwitz} for
the graph $G = \{12,23\}$. Two of the four terms of
 $[U_G] = 8 T_1 T_2 T_3 (T_1+T_2)(T_2+T_3)$ are displayed.
\end{example}

\begin{figure}[h]
\centering
\begin{tcolorbox}[size=fbox,width=0.5\linewidth,colback=blue!5!white,colframe=blue!75!black]
\begin{Verbatim}[fontsize=\scriptsize,commandchars=\\\{\}]
G = \{\{1,2\},\{2,3\}\};
netList multidegHurwitz(3,G)
     +-+-+------+-----------+
     | | | 2 2  |           |
o3 = |0|8|T T T |\{8, 16, 24\}|
     | | | 1 2 3|           |
     +-+-+------+-----------+
     | | |   3  |           |
     |1|8|T T T |\{16, 8, 16\}|
     | | | 1 2 3|           |
     +-+-+------+-----------+
        ... etc ... etc ...
\end{Verbatim}
\end{tcolorbox}
\caption{These two multigraded Hurwitz forms have degrees $(8,16,24)$ and $(16,8,16)$.
}\label{fig:multidegHurwitz}
\end{figure}

The goal of this paper has now been accomplished.
We completed the square in Figure~\ref{fig:22}
by introducing the multigraded Hurwitz form.
Sections \ref{sec5} and \ref{sec6} showed 
how this is useful for applied scenarios 
whose underlying geometry takes place in
a product of projective spaces.

\bigskip \bigskip
\bigskip 

\noindent {\bf Acknowledgements}:
We thank
Enrique Arrondo, Viktoriia Borovik and 
Clara Briand for helpful discussions.
Elizabeth Pratt was supported by US National Science Foundation (GRFP no.~2023358166).

\bigskip \medskip \bigskip

\noindent
\footnotesize {\bf Authors' addresses:}
\smallskip

\noindent Elizabeth Pratt, UC Berkeley \hfill \url{epratt@berkeley.edu}

\noindent Luca Sodomaco, MPI MiS Leipzig, Germany \hfill \url{luca.sodomaco@mis.mpg.de}

\noindent Bernd Sturmfels, MPI MiS Leipzig, Germany \hfill \url{bernd@mis.mpg.de}
	
\end{document}